\documentclass[11pt]{amsart}
\usepackage{amsfonts,amsmath,amssymb}
\usepackage{mathabx}
\usepackage{mathrsfs}
\usepackage[latin1]{inputenc}
\usepackage[usenames,dvipsnames]{pstricks}
\newtheorem{theorem}{Theorem}[section]

\newtheorem{Lemm}[theorem]{Lemma}

\newcommand{\p}{\partial}
\newcommand{\R}{\mathbb{R}}

\newcommand{\para}[1]{\left(#1\right)}

\usepackage{graphicx}

\usepackage{mathrsfs}
\usepackage[latin1]{inputenc}
\usepackage[T1]{fontenc}
\usepackage{color,xspace}
\usepackage{hyperref}
\usepackage[usenames,dvipsnames]{pstricks}
\usepackage{times}
\allowdisplaybreaks
\begin{document}
\title[An inverse problem of finding two time-dependent coefficients]{An inverse problem of finding two time-dependent coefficients in second order hyperbolic equations from Dirichlet to Neumann map}
\author[M.~Bellassoued]{Mourad~Bellassoued}
\author[I.~Ben A\"{\i}cha]{Ibtissem Ben A\"{\i}cha}
\address{M.~Bellassoued.  Universit\'e de Tunis El Manar, Ecole Nationale d'Ing\'enieurs de Tunis, LAMSIN, B.P. 37, 1002 Tunis, Tunisia.}
\email{mourad.bellassoued@enit.utm.tn }
\address{I.~Ben A\"{\i}cha. Universit\'e de Tunis El Manar, Ecole Nationale d'Ing\'enieurs de Tunis, LAMSIN, B.P. 37, 1002 Tunis, Tunisia.}
\email{ibtissem.benaicha@enit.utm.tn}
\date{\today}
\begin{abstract}
In the present paper, we consider a non self adjoint hyperbolic operator with a vector  field and an electric potential that depend not only on the space variable but also on the time variable. More precisely, we attempt to stably and simultaneously retrieve  the real valued velocity field and  the real valued potential from the knowledge of Neumann measurements performed on the whole  boundary of the domain. We establish in dimension $n$ greater than two,  stability estimates for the problem under consideration. Thereafter, by enlarging the set of data we show that the unknown terms can be stably retrieved in larger regions including the whole domain. The proof of the main results are mainly based on the reduction of the inverse problem under investigation  to an equivalent and classic inverse problem for an electro-magnetic wave equation. 
\end{abstract}
\maketitle 
\section{Introduction and  main results}\label{Section 1}
 Let $T>0$ and  $\Omega\subset \R^{n}$ with  $n\geq2$, be a  bounded domain with a sufficiently smooth boundary $\Gamma=\p\Omega$. A lot of physical phenomena can be described by partial differential equations and in this paper we are interested in the wave propagation phenomenon which is described by the following hyperbolic equation 
 $$ \mathcal{L}_{V,p}=\p_{t}^{2}-\Delta+V(x,t)\cdot \nabla+p(x,t)$$
 with a real valued time-space dependent velocity field $V=(V^1,...,V^n)\in\! \mathcal{C} ^{3}(Q,\R^{n})$. This equation is also disturbed by an electric potential  $p\in \mathcal{C}^{1}(Q,\R)$ which is a function of both variables: $x$ which is the spatial variable that is assumed to live in the bounded domain $\Omega$ and and the time variable $t\in (0,T)$. We denote by $Q=\Omega\times(0,T)$ the cylindrical domain of propagation and by $\Sigma=\Gamma\times(0,T)$ its lateral boundary.  To state things clearly, the main purpose of this paper is the study of the inverse problem of determining the two time-space dependent terms $V$ and $p$ from measurements made on the solution $u$ of the following system
\begin{equation}\label{1.1}
\left\{
  \begin{array}{ll}
   \mathcal{L}_{V,p}u=0 & \mbox{in} \,Q, \\
   \\
    u(\cdot,0)= u_0,\,\,\p_{t}u(\cdot,0)=u_1 & \mbox{in}\, \Omega,  \\
    \\
    u=f & \mbox{on} \,\Sigma,
  \end{array}
\right.
\end{equation}  
where   
$u_0\!\in\! H^1(\Omega)$ and $u_1\!\in\! L^2(\Omega)$ are the initial conditions and  $ f\!\in\!\mathscr{H}^{1}_{0}(\Sigma)\!:=\!\{ f\!\in\! H^{1}(\Sigma),\,\,f(\cdot,0)=u_{0|\Gamma}\}$ is a non homogeneous Dirichlet data that is used to probe the system. More precisely, we will focus on the stability issue. Namely, we hope to know weather the unknown terms depend or not, on the observed measurements.  Before dealing with the problem under consideration, let us state in brief some of the results that are relevant to this problem. 

The study of inverse coefficients problems for partial differential equations is one of the most rapidly growing mathematical research area in the recent years.  There is a wide mathematical literature on this issue, but it is mainly concerned with unknown  coefficients of order zero on space.  In \cite {[R23]} Rakesh and Symes proved the uniqueness for the problem of determining the time independent potential in the wave equation from the global Dirichlet to Neumann (DN) map based on the construction of geometric optics solutions. In  \cite{[R17]}, Eskin \cite{[R13]} and Isakov showed a unique determination of the potentiel from the knowledge of the local DN map.  Bellassoued, Choulli and Yamamoto in \cite{[R5]}  
treated the  stability  where the Neumann observations are on any arbitrary  sub-boundary.  Otherwise, in \cite{[13]} the authors showed a H\"older type stability estimate for the determination of a coefficient in a subdomain from the local DN map.  As regards stability from measurements made on the whole boundary, one can see 
 \cite{[R30]} and \cite{[R11]}. 

There are also many publications related to this kind of inverse problems in Riemannian case. We state for example the paper of Bellassoued and Dos  Santos Ferreira \cite{[R6]}, Stefanov and Uhlmann \cite{[R29]}.  Other than the mentioned papers, the recovery of time-dependent coefficients in hyperbolic equations has also been developped recently, we refer e.g to \cite{[B&I],[B&I2],[I],[Y]} and the references therein for reader's curiosity. 
  
 If the coefficient to be determined is of order $1$  on space, we quote for example the paper of  Pohjola \cite{[Poh]} who  studied the determination of a velocity field in a steady state convection diffusion equation from the DN map and  proved a uniqueness result for that problem. The same problem was considered in  \cite{[cheng]} where the uniqueness for coefficients of less regularity was proved . The case of Lipschitz continuous coefficients was seen by Salo \cite{[salo]}.  The stability issue was considered the first time by  Bellassoued and Ben A\"icha \cite{[new]}. They established the stable recovery a space-dependent velocity field in a non self adjoint hyperbolic equation.
 
 In this paper, we would like to know whether it is possible to recover the
velocity field $V$ and the electric potential $p$ in the cylindrical domain $Q$ by controlling the response of the medium on  the whole boundary  $\Sigma,$  after being probed  with non homogeneous  Dirichlet 
disturbances.

 We will head toward three different situations. We will change each time the considered set of measurements, and  in each case  stability estimates will be established  for the recovery of the unknown terms in different areas. In the first and 
  second cases, the initial data will be fixed to zero. To state things clearly, in the first case, the lateral observations are mathematically modeled by the so-called Dirichlet-to-Neumann map $\Lambda_{V,p}:\mathscr{H}^{1}_{0}(\Sigma)\longrightarrow L^{2}(\Sigma)$ associated with $\mathcal{L}_{V,p} $  with $(u_0,u_1)= (0,0)$ defined  by $\Lambda_{V,p}(f)=\p_{\nu}u$. 
  The vector  $\nu$ is the unit outward normal to the the boundary $\Gamma$ at the point $x$ and $\p_{\nu}u$ denotes the quantity  $\nabla u\cdot \nu$. More precisely, we would like to know if a small perturbation on the  boundary measurement $\Lambda_{V,p}$ can cause an error in the determination of $V$ and $p$. It seems that this paper is the first dealing with the simultaneous recovery of time and space dependent terms of order one and zero on space appearing in a non self adjoint hyperbolic operator. 

 Inspired by the work of \cite{[new]},  we firstly adress the stability issue for the inverse problem under investigation and we show that the velocity field $V$ and the electric potential $p$ stably depend on the DN map $\Lambda_{V,p}$ and we establish a stability estimate of $\log$ type for the recovery of $V$ and a stability of $\log$-$\log$ type for the recovery of $p$ via the DN map    $\Lambda_{V,p}$ in a precise region of the cylindrical domain $Q$ provided that they are known outside of this region.  Note that in the case where the initial conditions are fixed to zero and in light of  domain of dependence arguments, it is hopeless to recover  time-dependent coefficients in hyperbolic operators everywhere, even from the knowledge of global Neumann observations because  the value of the solution depend on the value of the initial data.  For more details about this issue, one can see \cite{[R15]}.

   In order to state our main results, we need  first to set up some notations and terminologies that will be used in the rest of the paper. Let $r>0$ be such that $T>2r$. We assume that the spatial domain $\Omega$ satisfies 
$$\Omega\subseteq B\Big(0,\frac{r}{2}\Big)=\Big\{x\in
\R^{n},\,|x|<\frac{r}{2}\Big\}.$$
 We denote by  $Q_{r}=B(0,r/2)\times(0,T)$ the domain of propagation  and by $\mathscr{C}_r$  the follwonig subdomain 
 $$\mathscr{C}_r=\left\{x\in\R ^{n},\,\,\displaystyle\frac{r}{2}<|x|<T-\displaystyle\frac{r}{2}\right\}.$$
On the other hand, we denote by $\mathcal{F}_r$  the following forward light cone  defined as follows
$$\mathcal{F}_{r}=\left\{(x,t)\in Q_{r},\,\,|x|<t-\displaystyle\frac{r}{2},\,t>\frac{r}{2}\right\},$$
and by $\mathcal{B}_r$ the backward light cone  given by the following set
$$\mathcal{B}_r=\left\{(x,t)\in Q_{r},\,\,|x|<T-\displaystyle\frac{r}{2}-t,\,T-\frac{r}{2}> t\right\}.$$
Let us define $\mathcal{I}_r=\mathcal{F}_r\cap \mathcal{B}_r$ and  $\mathcal{I}_r^*= Q\cap \mathcal{I}_r.$  We notice here that when $\Omega=B(0,r/2),$ one gets  $\mathcal{I}_r^*=\mathcal{I}_r$. 
For  $M\!>\!0$,   let us denote 
$$\mathcal{S}\,(M)=\Big\{(V,p)\in \mathcal{C}^{3}(Q,\R^{n})\times \mathcal{C}^1(Q,\R),\,\,\|V\|_{{W}^{3,\infty}(Q)}+\|p\|_{W^{1,\infty}(Q)}\leq M \Big\}.$$
\begin{theorem}\label{Theorem 1.1}
For $T\!\!>\!2\,Diam\, (\Omega)$, there exist  $C,\mu>0$ such that 
if  $\|\Lambda_{V_{1},p_{1}}-\Lambda_{V_{2},p_{2}}\|\leq m\in(0,1),$ then 
$$\|V_{1}-V_{2}\|_{L^{\infty}(\mathcal{I}_r^*)}\leq C \,|\log\|\Lambda_{V_{2},p_{2}}-\Lambda_{V_{1},p_{1}}\||^{-\mu},$$
 for all $(V_{j}, p_j)\in\!\mathcal{S}(M)$ satisfying $\|V_j\|_{H^{\alpha}(Q)}+\|p_j\|_{H^{\alpha}(Q)}\leq M,\, j=1,2,$ for some  $\alpha\!\!>\!\!n/2+1$ while assuming  that  $(V_1,p_1)\!\!=\!\!(V_2,p_2)$ in $\overline{Q}_{r}\setminus \mathcal{I}_r^*$ and $\mbox{div$_{x}$} (V_{1})=\mbox{div$_{x}$} (V_{2})$.  Moreover, there exist $C',\mu'>0$ such that   
 $$\|p_2-p_1\|_{L^{\infty}(\mathcal{I}_r^*)}\leq C' \Big( \log |\log \|\Lambda_{V_{2},p_{2}}-\Lambda_{V_{1},p_{1}}\||\Big)^{-\mu'}, $$
holds true. Here $C$ and $C'$ depend only on $\Omega$ and $M$.
\end{theorem}

The above theorem states the simultaneous  determination of the velocity field  $V$ and the electric potential $p$ from the knowledge of the DN map $\Lambda_{V,p}$  in $\mathcal{I}_r^*\subset Q$, provided they are known outside of this subset. To improve the previous statement, we need to know more about $u$,  solution to (\ref{1.1}). In the second  part of the paper, $(u_0,u_1)$ are as usual \textit{frozen} to zero and our observations are made by  the linear response operator $\mathcal{R}_{V,p}:\mathscr{H}^{1}_{0}(\Sigma)\longrightarrow  \mathcal{K}:=L^{2}(\Sigma)\times H^1(\Omega)\times L^2(\Omega)$ given by $\mathcal{R}_{V,p}(f)= (\p_{\nu}u,u(\cdot, T),\p_tu(\cdot,T)).$

Note that even by adding the final data of the solution $u$ of the wave equation (\ref{1.1}), it is impossible to determine the unknown terms everywhere in $Q$ because we still have $(u_0,u_1)=(0,0)$.  In the sequel, we denote $\mathcal{I}_r^\sharp= Q\cap \mathcal{F}_r$.
\begin{theorem}\label{Theorem 1.2}
For $T\!\!>\!2\,{Diam}\, (\Omega)$, There exist  $C,\mu>0$ such that 
if  $\|\mathcal{R}_{V_{1},p_{1}}-\mathcal{R}_{V_{2},p_{2}}\|\leq\! m\in(0,1),$ then 
$$\|V_{1}-V_{2}\|_{L^{\infty}(\mathcal{I}_r^\sharp)}\leq C \,|\log\|\mathcal{R}_{V_{2},p_{2}}-\mathcal{R}_{V_{1},p_{1}}\||^{-\mu},$$
 for all  $(V_{j}, p_j)\in\!\mathcal{S}(M)$ satisfying $\|V_j\|_{H^{\alpha}(Q)}+\|p_j\|_{H^{\alpha}(Q)}\leq M,\, j=1,2.$ for some  $\alpha\!\!>\!\!n/2+1$ while assuming  that  $(V_1,p_1)\!\!=\!\!(V_2,p_2)$ in $\overline{Q}_{r}\setminus \mathcal{I}_r^\sharp$ and $\mbox{div}_{x} (V_{1})=\mbox{div}_{x} (V_{2})$.  Moreover, there exist $C',\mu'>0$ such that  
 $$\|p_2-p_1\|_{L^{\infty}(\mathcal{I}_r^\sharp)}\leq C' \Big( \log |\log \|\mathcal{R}_{V_{2},p_{2}}-\mathcal{R}_{V_{1},p_{1}}\||\Big)^{-\mu'}, $$
holds true. Here $C$ and $C'$ depend only on $\Omega$ and $M$.
\end{theorem}
In the two previous cases, we notice that  it is impossible to identify the  unknown terms over the entire cylinder $Q$ and to do this we need to consider the data of the second case and vary the initial conditions. 
We introduce this space $E= \mathscr{H}^1_0(\Sigma)\!\times\! H^1(\Omega)\!\times \!L^2(\Omega).$ In this last case the  measures are modeled  by 
$
\Gamma_{V,p}:{E}\longrightarrow \mathcal{K}$
which is defined by  
$\Gamma_{V,p}((f,u_{0},u_{1}))= (\p_{\nu}u,u(\cdot,T),\p_{t}u(\cdot,T))$.

Thus, by these extra information about $u$ of the hyperbolic equation (\ref{1.1}), we can see that  is possible to identify $V$ and $p$ everywhere in the cylinder $Q$.
\begin{theorem}\label{Theorem 1.3}
 There exist  $C,\mu>0$ such that 
if  $\|\Gamma_{V_{1},p_{1}}-\Gamma_{V_{2},p_{2}}\|\leq\! m\in(0,1),$ then we have 
$$\|V_{1}-V_{2}\|_{L^{\infty}(Q)}\leq C \,|\log\|\Gamma_{V_{2},p_{2}}-\Gamma_{V_{1},p_{1}}\||^{-\mu},$$
 for all $(V_{j}, p_j)\in\!\mathcal{C}^3(\overline{Q})\times \mathcal{C}^1(\overline{Q})$ satisfying $\|V_j\|_{H^{\alpha}(Q)}\!+\!\|p_j\|_{H^{\alpha}(Q)}\leq M,\, j=1,2.$ for some  $\alpha\!\!>\!\!n/2+1$ while assuming  that  $\mbox{div}_{x} (V_{1})=\mbox{div}_{x} (V_{2})$.  Moreover, there exist $C',\mu'>0$ such that   
 $$\|p_2-p_1\|_{L^{\infty}(Q)}\leq C' \Big( \log |\log \|\Gamma_{V_{2},p_{2}}-\Gamma_{V_{1},p_{1}}\||^{\mu'}\Big)^{-1}, $$
holds true. Here $C$ and $C'$ depend only on $\Omega$ and $M$.
\end{theorem}

Note that the considered response operators are linear and continuous. We point out that  to be able to prove these results, we shall first reduce the problem under consideration to an equivalent  one that we are familiar with and that is easier to deal with. Namely, we need first to deal with a preliminary  problem  for another hyperbolic operator. More details about this issue will be given in the next section. 

The rest of this text is organized as follows. In Section \ref{Section 2}, we reduce the problem under investigation to an equivalent problem for a wave equation with both  magnetic and electric potentials. In Section \ref{Section 3}, we deal with the auxiliary inverse problem and we give the proof of Theorem \ref{Theorem 1.1}. In section \ref{Section 4}, we show that $V$ and $p$ can be stably recovered in a bigger subdomain by improving the  data set. In Section \ref{Section 5} we establish  Theorem \ref{Theorem 1.3}. 
\section{Reduction of the problem}\label{Section 2}
This section is dedicated to reduce the problem under investigation. 
The idea behind  proving the stability in the  identification of  $V$ and $p$ in (\ref{1.1}) is essentially based on treating a classic inverse problem associated with this  equation
$$ \mathcal{H}_{A,q}=\p_{t}^{2}-\Delta_{A}+q(x,t),$$
where $A=(a_1,...,a_n)\in {W}^{3,\infty}(Q,\mathbb{C}^{n})$ is  an imaginary  complex magnetic potential (pure), $q\in W^{1,\infty}(Q,\R)$ is the electric potential that is assumed to be a real bounded function  and $\Delta_{A}$ here is the magnetic Laplacien that one defines it as:
 $$\Delta_{A}=\sum_{j=1}^{n}(\p_{j}+i A_{j})^{2}=\Delta +2i A \cdot \nabla+i\,\mbox{div}_x\,A-A\cdot A.$$
 We consider the following equation 
\begin{equation}\label{2.2}
\left\{
  \begin{array}{ll}
   \mathcal{H}_{A,q}u=0 & \mbox{in} \,Q, \\
   \\
   u(\cdot,0)=u_0,\,\,\p_t u(\cdot,0)=u_1 & \mbox{in}\, \Omega,  \\
   \\
    u=f & \mbox{on} \,\Sigma,
  \end{array}
\right.
\end{equation}
where $(u_0,u_1)$ are the initial conditions   and $f\in\mathscr{H}^{1}_{0}(\Sigma)$ is the non homogeneous Dirichlet data that is used as a stimulation term. 

The inverse problem associated with the equation (\ref{1.1}) and consisting on recovering $V$ and $p$ from the knowledge of the different data sets prescribed before,  may equivalently be reformulated to the inverse problem of recovering the magnetic potential $A$ and the electric potential $q$ appearing in (\ref{2.2}) from equivalent measurements. This is  actually feasible if one writes $A$ and $q$ in precise forms  in such a way $\mathcal{H}_{A,q}=\mathcal{L}_{V,p}$  and the corresponding response operators coincide. 

\noindent The forward problem related to (\ref{2.2}) is well posed (see \cite{[LM]}), therefore we may introduce the DN map $\widetilde{\Lambda}_{A,q}:\mathscr{H}^{1}_{0}(\Sigma) \longrightarrow L^{2}(\Sigma)$ associated to (\ref{2.2}) with $(u_0,u_1)=(0,0)$ defined by  
 $
f\longmapsto (\p_{\nu}+iA\cdot \nu)u.$

\noindent At an initial stage, our goal is to establish Theorem \ref{Theorem 1.1} that claims the stable recovery of $V$ and $p$ from $\Lambda_{V,p}$ which equivalently amounts to  showing that $A$ and $q$ in (\ref{2.2}) can be determined from $\widetilde{\Lambda}_{A,q}$.  Next, to be able to prove Theorem \ref{Theorem 1.2}, we just need to stably determine $A$  and $q$ in a bigger subset from measures \textit{ enclosed } in the equivalent operator  $\widetilde{\mathcal{R}}_{A,q}:\mathscr{H}^{1}_{0}(\Sigma)\longrightarrow \mathcal{K}$  associated to (\ref{2.2}) with $(u_0,u_1)=(0,0)$ that is defined by $
f\longmapsto ((\p_{\nu}+iA\cdot\nu)u,u(\cdot, T),\p_tu(\cdot,T)).$ 
 In the final case our objective is to prove Theorem \ref{Theorem 1.3}. To do this, we first show that $A$ and $q$ can be identified over the entire $Q$ if we just vary the initial data. To this end, we introduce the operator $\widetilde{\Gamma}_{A,q}:{E}\longrightarrow \mathcal{K}$ that is given by $$\widetilde{\Gamma}_{A,q}((f,u_{0},u_{1}))= (\p_{\nu}u+i A\cdot\nu u,u(\cdot,T),\p_{t}u(\cdot,T)).
$$
In the sequel of the text, we use these notations
  $$
\widetilde{{\Gamma}}^{1}_{A,q}(f,u_0,u_1):=(\p_{\nu}+i A\cdot\nu)u,\,\,\,\,\,\widetilde{{\Gamma}}^{2}_{A,q}(f,u_0,u_1):=u(\cdot,T),\,\,\,\,\,\,\widetilde{{\Gamma}}_{A,q}^{3}(f,u_0,u_1):=\p_{t}u(\cdot,T).
$$ 

\begin{Lemm}\label{Lemma 2.1} 
 Let $V_{j}\in W ^{3,\infty}(Q,\R^n)$ and $p_j\in L^{\infty}(Q,\R)$, $j=1,\,2$ be  given  such that  $V_1 \cdot \nu= V_2\cdot \nu$ on $\Gamma$ \color{black}. We set 
 \begin{equation}\label{2.3}
  A_{j}=\frac{i}{2}V_{j},\quad 
 \mbox{and} \quad q_{j}=p_{j}+\frac{1}{4}{V_{j}^{2}}-\frac{1}{2}\,\mbox{div$_{x}$}\, (V_{j}).
 \end{equation}
 Then, we have 
$
 \mathcal{H}_{A_{j},q_{j}}=\mathcal{L}_{V_{j}}$ and $\mathcal{H}^{*}_{A_{j},q_{j}}=\mathcal{L}^{*}_{V_{j}}$  for $ j=1,\,2.
 $  
 Moreover,  the following identities hold
 
\begin{equation}\label{2.4}
\|\widetilde{\Gamma}_{A_{1},q_{1}}-\widetilde{\Gamma}_{A_{2},q_{2}}\|=\|\Gamma_{V_{1},p_1}-\Gamma_{V_{2},p_2}\|.
\end{equation}
 Respectively, we have 
  \begin{equation}\label{2.5}
 \|\widetilde{\mathcal{R}}_{A_1,q_1}-\widetilde{\mathcal{R}}_{A_2,q_2}\|=\|\mathcal{R}_{V_1,p_1}-\mathcal{R}_{V_2,p_2}\|\,\,\, \mbox{and}\,\,\,\,\,\|\widetilde{\Lambda}_{A_{1},q_{1}}-\widetilde{\Lambda}_{A_{2},q_{2}}\|=\|\Lambda_{V_{1},p_1}-\Lambda_{V_{2},p_2}\|.\end{equation}

 \noindent Here $\|\cdot\|$ stands for the norm in $\mathscr{L}(E,\mathcal{K})$, (resp., $\mathscr{L}(\mathscr{H}^1_0(\Sigma),\mathcal{K})$ and $\mathscr{L}(\mathscr{H}^1_0(\Sigma), L^2(\Sigma))$).
\end{Lemm}
 \begin{proof}
 From (\ref{2.3}), one can  check that $
 \mathcal{H}_{A_{j},q_{j}}=\mathcal{L}_{V_{j}}$ and $\mathcal{H}^{*}_{A_{j},q_{j}}=\mathcal{L}^{*}_{V_{j}}$  for $ j=1,\,2$. In order to show (\ref{2.4}), let us  denote by  $w_{j}$  the solution of this equation
\begin{equation}\label{2.6}
\left\{
  \begin{array}{ll}
  \mathcal{L}_{V_{j},p_j}w_{j}=0 & \mbox{in} \,\,\,Q, \\
  \\
    w_{j}(\cdot,0)=u_0,\,\,\p_{t}w_{j}(\cdot,0)=u_1 & \mbox{in}\,\,\, \Omega,  \\
   \\
    w_{j}=f & \mbox{on} \,\,\,\Sigma,
  \end{array}
\right.
\end{equation} 
  and let  $v_{j}$ be the solution of this equation
  \begin{equation}\label{2.6*}
  \left\{
  \begin{array}{ll}
  \mathcal{L}^{*}_{V_{j},p_j}v_{j}=0 & \mbox{in} \,\,\,Q, \\
  \\
    v_{j}(\cdot, T)=u_2,\,\, \p_{t}v_{j}(\cdot,T)=u_3& \mbox{in}\,\,\, \Omega,  \\
   \\
    v_{j}=g & \mbox{on} \,\,\,\Sigma.
  \end{array}
\right.
\end{equation} 
\noindent Here $f,\,g\in H^1(\Sigma)$. On the other hand, for $j=1,2$,  we denote by  
$$
{{\Gamma}}^{1}_{V_j,p_j}(f,u_0,u_1):=\p_{\nu}w_j,\,\,\,\,\,{{\Gamma}}^{2}_{V_j,p_j}(f,u_0,u_1):=w_j(\cdot,T),\,\,\,\,\,\,{{\Gamma}}_{A_j,q_j}^{3}(f,u_0,u_1):=\p_{t}w_j(\cdot,T).
$$ 
\noindent We multiply (\ref{2.6}) with $\overline{v}_{j}$, then an integration by parts yields
\begin{multline}\label{2.7}
\int_{\Sigma}\Gamma^1_{V_{j},p_j}(f,u_0,u_1) \,\overline{g}\,d\sigma_x\,dt+\int_{\Omega}\Gamma^2_{V_j,p_j}(f,u_0,u_1) \,\overline{u}_3\,dx-\int_{\Omega}\Gamma_{V_j,p_j}^3(f,u_0,u_1) \,\overline{u}_2 \,dx\cr
=\int_{Q}\Big( \p_{t}^2\overline{v}_{j}\,w_j+\nabla w_{j}\cdot \nabla \overline{v}_{j}+V_{j}\cdot \nabla w_{j}\,\overline{v}_{j}+p_j w_j \overline{v}_j\Big) \,dx\,dt\cr
+\int_{\Omega} \Big( u_0\,\p_t \overline{v}_j(\cdot,0)-u_1\,\overline{v}_j(\cdot,0)\Big)\,dx.
 \end{multline}
 It is clear that   $w_{j}$  and $v_{j}$ with  $j=1,\,2$, are also solutions to 
 
 \begin{equation}\label{2.8}
\left\{
  \begin{array}{ll}
  \mathcal{H}_{A_{j},q_{j}}w_{j}=0 & \mbox{in} \,\,\,Q, \\
  \\
    w_{j}(\cdot,0)=u_0,\,\,\p_{t}w_{j}(\cdot,0)=u_1 & \mbox{in}\,\,\, \Omega,  \\
   \\
    w_{j}=f & \mbox{on} \,\,\,\Sigma,
  \end{array}
\right.; \quad\quad \left\{
  \begin{array}{ll}
   \mathcal{H}^{*}_{A_{j},q_{j}}v_{j} =0 & \mbox{in} \,\,\,Q, \\
  \\
    v_{j}(\cdot,T)=u_2,\,\,\p_{t}v_{j}(\cdot,T)=u_3 & \mbox{in}\,\,\, \Omega,  \\
   \\
    v_{j}=g & \mbox{on} \,\,\,\Sigma.
  \end{array}
\right.
\end{equation}

\noindent Then,  after multiplying the first equation in (\ref{2.8}) with $\overline{v}_{j}$ and by  integrating, we obtain
\begin{multline}\label{2.9}
\int_{\Sigma}\widetilde{\Gamma}^1_{A_{j},q_j}(f,u_0,u_1) \,\overline{g}\,d\sigma_x\,dt+\int_{\Omega}\widetilde{\Gamma}^2_{A_j,q_j}(f,u_0,u_1) \,\overline{u}_3\,dx-\int_{\Omega}\widetilde{\Gamma}_{A_j,q_j}^3(f,u_0,u_1) \,\overline{u}_2 \,dx\cr
=\int_{Q}\,\Big( \,\p_{t}^2\overline{v}_{j}\,w_j\,+\,\nabla w_{j}\cdot \nabla \overline{v}_{j}\,+\,V_{j}\cdot \nabla w_{j}\,\overline{v}_{j}\,+\,p_j\, w_j\, \overline{v}_j\Big) \,dx\,dt\cr
+\int_{\Omega}\!\Big(\! u_0\,\p_t \overline{v}_j(\cdot,0)-u_1\,\overline{v}_j(\cdot,0)\Big)dx -\frac{1}{2}\int_{\Sigma}V_j\cdot \nu\,w_j\,\overline{v}_jd\sigma_x dt.
 \end{multline}
 Here $d\sigma_x$ is the Euclidean surface measure on $\Gamma$. Therefore, in light of (\ref{2.7}) and (\ref{2.9}) one gets  
\begin{multline*}
\int_{\Sigma}\widetilde{\Gamma}^1_{A_{j},q_j}(f,u_0,u_1) \,\overline{g}\,d\sigma_x\,dt+\int_{\Omega}\widetilde{\Gamma}^2_{A_j,q_j}(f,u_0,u_1) \,\overline{u}_3\,dx-\int_{\Omega}\widetilde{\Gamma}_{A_j,q_j}^3(f,u_0,u_1) \,\overline{u}_2 \,dx\cr
 = \int_{\Sigma}\Gamma^1_{V_{j},p_j}(f,u_0,u_1) \,\overline{g}\,d\sigma_x\,dt+\int_{\Omega}\Gamma^2_{V_j,p_j}(f,u_0,u_1) \,\overline{u}_3\,dx\cr
  -\int_{\Omega}\Gamma_{V_j,p_j}^3(f,u_0,u_1) \,\overline{u}_2 \,dx  -\frac{1}{2}\int_{\Sigma}V_j\cdot \nu\,w_j\,\overline{v}_jd\sigma_x dt.
 \end{multline*}
Since $V_{1}\cdot\nu =V_{2}\cdot \nu$ on $\Gamma$, one gets
\begin{multline*}
\int_{\Sigma}(\widetilde{\Gamma}^1_{A_{1},q_1}-\widetilde{\Gamma}^1_{A_2,q_2})(f,u_0,u_1) \,\overline{g}\,d\sigma_x\,dt+\int_{\Omega}(\widetilde{\Gamma}^2_{A_2,q_2}-\widetilde{\Gamma}^2_{A_1,q_1})(f,u_0,u_1) \,\overline{u}_3\,dx\cr
-\int_{\Omega}(\widetilde{\Gamma}_{A_1,q_2}^3-\widetilde{\Gamma}^3_{A_1,q_1}(f,u_0,u_1) \,\overline{u}_2 \,dx = \int_{\Sigma}(\Gamma^1_{V_{2},p_2}-{\Gamma}^1_{V_1,p_1})(f,u_0,u_1) \,\overline{g}\,d\sigma_x\,dt\cr
+\int_{\Omega}(\Gamma^2_{V_1,p_1}-\Gamma^2_{V_2,p_2})(f,u_0,u_1)) \,\overline{u}_3\,dx -\int_{\Omega}(\Gamma_{V_1,p_1}^3-\Gamma_{V_2,p_2}^3)(f,u_0,u_1) \,\overline{u}_2 \,dx.
 \end{multline*}
Using the fact that  $\widetilde{\Gamma}^j_{A_j,q_j}(f,u_0,u_1)=\Gamma^j_{V_j,p_j}(f,u_0,u_1)$ for $j=2,\,3$, one can thus easily see that  $$\|\widetilde{\Gamma}_{A_1,q_1}-\widetilde{\Gamma}_{A_2,q_2}\|=\|\Gamma_{A_1,q_1}-\Gamma_{A_2,q_2}\|.$$
As a consequence , in the particular case where the initial conditions $u_0$ and $u_1$ are zero  we have  for $j=1,2$.
$$\mathcal{R}_{V_j,p_j}(f)=\Gamma_{V_j,p_j}(f,0,0)\quad \mbox{and}\quad \widetilde{\mathcal{R}}_{A_j,q_j}(f)=\widetilde{\Gamma}_{A_j,q_j}(f,0,0).$$
Thus, from what precedes, we conclude that 
$$\|\widetilde{\mathcal{R}}_{A_1,q_1}-\widetilde{\mathcal{R}}_{A_2,q_2}\|=\|\mathcal{R}_{V_1,p_1}-\mathcal{R}_{V_2,p_2}\|. $$
By a similar way, one can prove the last identity (for the convenience of the reader, the proof of this identity is given in  \cite{[new]}).
 \end{proof}

 It is well known that in the case where  $(u_0,u_1)=(0,0)$ the Dirichlet to Neumann map is invariant under a gauge transformation of the magnetic potential (see \cite{[S]}).  Therefore, the magnetic potential $A$ can not be uniquely determined from the knowledge of Neumann measurements.  However,  we could hope to reconstruct the \textit{divergence known} magnetic potential.

\noindent We move to state our preliminary result. But before this, let us introduce the set $\widetilde{\mathcal{S}}(M)$,for any  $M>0$.  
$$\widetilde{\mathcal{S}}(M):=\{(A,q)\in \mathcal{C}^{3}(Q,\mathbb{C}^{n})\times\mathcal{C}^{1}(Q,\R),\,\, \|A\|_{{W}^{3,\infty}(Q)}+\|q\|_{{W}^{1,\infty}(Q)}\leq M \}.$$
Then, this claim is true:
\begin{theorem}\label{Theorem 2.2}
We assume that $T\!\!>\!2\,Diam\, (\Omega)$. There exist  $C,\mu>0$ such that 
if  $\|\widetilde{\Lambda}_{A_{1},q_{1}}-\widetilde{\Lambda}_{A_{2},q_{2}}\|\leq m\in(0,1),$ then 
$$\|A_{1}-A_{2}\|_{L^{\infty}(\mathcal{I}_r^*)}\leq C \,|\log\|\widetilde{\Lambda}_{A_{2},q_{2}}-\widetilde{\Lambda}_{A_{1},q_{1}}\||^{-\mu},$$
 for any  $(A_{j}, q_j)\in\!\widetilde{\mathcal{S}}(M)$ satisfying $\|A_j\|_{H^{\alpha}(Q)}+\|q_j\|_{H^{\alpha}(Q)}\leq M,\, j=1,2,$ for some  $\alpha\!\!>\!\!n/2+1$ while assuming  that  $(A_1,q_1)\!\!=\!\!(A_2,q_2)$ in $\overline{Q}_{r}\setminus \mathcal{I}_r^*$ and $\mbox{div$_{x}$} (A_{1})=\mbox{div$_{x}$} (A_{2})$.  Moreover, there exist $C',\mu'>0$ such that   
 $$\|q_2-q_1\|_{L^{\infty}(\mathcal{I}_r^*)}\leq C' \Big( \log |\log \|\widetilde{\Lambda}_{A_{2},q_{2}}-\widetilde{\Lambda}_{A_{1},q_{1}}\||\Big)^{-\mu'}, $$
holds true. Here $C$ and $C'$ depend only on $\Omega$ and $M$.
\end{theorem}
\noindent Similarly, we can reduce Theorem \ref{Theorem 1.2} to:
\begin{theorem}\label{Theorem 2.3}
We assume that  $T\!\!>\!2\,Diam\, (\Omega)$. There exist  $C,\mu>0$ such that 
if  $\|\widetilde{\mathcal{R}}_{A_{1},q_{1}}-\widetilde{\mathcal{R}}_{A_{2},q_{2}}\|\leq\! m\in(0,1),$ then 
$$\|A_{1}-A_{2}\|_{L^{\infty}(\mathcal{I}_r^\sharp)}\leq C \,|\log\|\widetilde{\mathcal{R}}_{A_{2},q_{2}}-\widetilde{\mathcal{R}}_{A_{1},q_{1}}\||^{-\mu},$$
 for any  $(A_{j}, q_j)\in\!\widetilde{\mathcal{S}}(M)$ satisfying $\|A_j\|_{H^{\alpha}(Q)}+\|q_j\|_{H^{\alpha}(Q)}\leq M,\, j=1,2,$ for some  $\alpha\!\!>\!\!n/2+1$ while assuming  that  $(A_1,q_1)\!\!=\!\!(A_2,q_2)$ in $\overline{Q}_{r}\setminus \mathcal{I}_r^\sharp$ and $\mbox{div}_{x} (A_{1})=\mbox{div}_{x} (A_{2})$.  Moreover, there exist $C',\mu'>0$ s. t  
 $$\|q_2-q_1\|_{L^{\infty}(\mathcal{I}_r^\sharp)}\leq C' \Big( \log |\log \|\widetilde{\mathcal{R}}_{A_{2},q_{2}}-\widetilde{\mathcal{R}}_{A_{1},q_{1}}\||\Big)^{-\mu'}, $$
holds true. Here $C$ and $C'$ depend only on $\Omega$ and $M$.
\end{theorem}
\noindent And if we further vary $u_0$ and $u_1$, we will then be able to show the \textit{full} recovery of $V$ and $p$. This amounts to proving the \textit{full }recovery of the known divergence potential $A$ and $q$. 
\begin{theorem}\label{Theorem 2.4}
 There exist  $C,\mu>0$ such that 
if  $\|\widetilde{\Gamma}_{A_{1},q_{1}}-\widetilde{\Gamma}_{A_{2},q_{2}}\|\leq\! m\in(0,1),$ then we have 
$$\|A_{1}-A_{2}\|_{L^{\infty}(Q)}\leq C \,|\log\|\widetilde{\Gamma}_{A_{2},q_{2}}-\widetilde{\Gamma}_{A_{1},q_{1}}\||^{-\mu},$$
 for all  $(A_{j}, q_j)\in\!\mathcal{C}^3(\overline{Q})\!\!\times\!\!\mathcal{C}^1(\overline{Q})$ satisfying $\|A_j\|_{H^{\alpha}(Q)}\!+\!\|q_j\|_{H^{\alpha}(Q)}\leq M,\, j=1,2,$ with $\alpha\!\!>\!\!n/2+1$ while assuming  that  $\mbox{div}_{x} (A_{1})=\mbox{div}_{x} (A_{2})$.  Moreover, there exist $C',\mu'>0$ such that   
 $$\|q_2-q_1\|_{L^{\infty}(Q)}\leq C' \Big( \log |\log \|\widetilde{\Gamma}_{A_{2},q_{2}}-\widetilde{\Gamma}_{A_{1},q_{1}}\||\Big)^{-\mu'}, $$
holds true. Here $C$ and $C'$ depend only on $\Omega$ and $M$.
\end{theorem}

The previous theorems describe the stable recovery of the magnetic potential of known divergence  $A$  and the electric potential $q$ in different areas from the knowledge of different types of measurements. In the next sections we prove  these preliminary results.
\section{Determination  of the unknown terms from boundary observations }\label{Section 3}
In this section we focus on proving our first main result. Namely we aim to stably recover $V$ and $p$ appearing in (\ref{1.1}) from  $\Lambda_{V,p}$ which amounts to stably recover  $A$ and $q$ appearing in (\ref{2.2}) from $\widetilde{\Lambda}_{A,q}$. The latter inverse problem is related to the one seen by R. Salazar and A. Waters  \cite{[Salazar]} who  tried to recover time-space dependent, vector and scalar potentials in the relativistic Schr\"odinger equation  in an infinite cylindrical domain $\Omega\times\R$.  We shall first construct  special solutions to the equation (\ref{2.2}).
\subsection{Construction of geometrical optics solutions}
 Now we are going to  construct particular solutions which plays an important role in proving the results. Let $\varphi\in \mathcal{C}_{0}^{\infty}(\R^{n})$. For any $\omega\in \mathbb{S}^{n-1}$, we denote by
\begin{equation}
\label{3.10}
\phi(x,t)=\varphi(x+t\omega).
\end{equation}
We can check that $\phi$ is a solution to 
$(\p_{t}-\omega\cdot \nabla)\phi(x,t)=0.$
We next establish the following lemma.  
\begin{Lemm}\label{Lemma 3.1}  Let  $\omega\in\mathbb{S}^{n-1}$ and $\varphi\in \mathcal{C}^{\infty}_{0}(\R^{n})$. Let  $\phi$ be given by (\ref{3.10}).  Then, for  any $\sigma>0$ the equation $\mathcal{H}_{A,q}u=0$ in $Q$ admits a solution  of the form 
$$u(x,t)=\phi(x,t)\,b(x,t)\,e^{i\sigma(x\cdot \omega+t)}+r(x,t).$$
Here $b$ is given by 
$$b(x,t)=\exp \Big(i\int_{0}^{t} \omega\cdot A (x+(t-s)\omega,s)\,ds\Big),$$
and the correction term $r$ satisfies
$r_{t=0}=\p_{t}r_{|t=0}=0$ in $\Omega$,  and  $r_{|\Sigma}=0$.
In addition, we have 
\begin{equation}\label{3.11}
\sigma \|r\|_{L^{2}(Q)}+\|\nabla r\|_{L^{2}(Q)}\leq C \|\varphi\|_{H^{3}(\R^{n})},
\end{equation}
where $C$ is a positive constant.
\end{Lemm}
\begin{proof}{}
 Setting 
 $$g(x,t)=-\mathcal{H}_{A,q}(\phi(x,t)b(x,t)e^{i\sigma(x \cdot\omega +t)}).$$  To establish our lemma we just need to show that there exists $r$ that satisfies \begin{equation}\label{3.12}
    \mathcal{H}_{A,q}r=g, \,\,\mbox{in}\,\,Q,  \quad
  r_{|t=0}=\p_{t}r_{|t=0}=0,\,\,\mbox{in}\,\,\Omega,\quad \mbox{and}\,\,\,\,\,
    r_{|\Sigma}=0,   
\end{equation}
in light of (\ref{3.10}) and since $b$ solves $ (\p_{t}-\omega\cdot\nabla -i\omega\cdot A )b=0,$ one gets 
$$
g(x,t)=-e^{i\sigma(x\cdot\omega+t)}\mathcal{H}_{A,q}(\phi(x,t)b(x,t) )
=-e^{i\sigma(x\cdot\omega+t)}g_{0}(x,t).
$$
Here $g_{0}\in L^{1}(0,T,L^{2}(\Omega))$. We denote by $w$ this fucntion 
\begin{equation}\label{3.13}
\tilde{w}(x,t)=\int_{0}^{t}r(x,s)\,ds.
\end{equation}
Thus, from (\ref{3.13}) and after integrating (\ref{3.12}) over $[0,t]$ for any $t\in(0,T)$, we obtain 
$\mathcal{H}_{A,q}\tilde{w}=F_{1}+F_{2} \,\,\mbox{in}\,Q,$ with $
   \tilde{w}_{|t=0}=\p_{t}\tilde{w}_{|t=0}=0 $ and 
    $\tilde{w}_{| \Sigma}=0$. 
Here $F_1$  and $F_{2}$ is defined as follows
$$F_{1}(x,t)=\displaystyle\int_{0}^{t}g(x,s)\,ds,$$ 
\begin{multline*}
F_{2}(x,t)=\int_{0}^{t}\Big(A^{2}(x,t)-A^{2}(x,s)\Big)r(x,s)\,ds+\int_{0}^{t}\Big( q(x,t)-q(x,s) \Big)  r(x,s)\,ds \cr
-i\int_{0}^{t} \Big(\mbox{div$_{x}$} A(x,t) -\mbox{div$_{x}$} A(x,s)\big) r(x,s)\,ds -2i \int_{0}^{t}\Big(  A(x,t)- A(x,s)\Big) \cdot \nabla r(x,s)\,ds. 
\end{multline*}
Next, we apply the well known energy estimate designed for hyperbolic equations on the interval
$[0,\tau]$ for  $\tau\in [0,T]$, then we obtain 
$$
\|\p_{t}\tilde{w}(\cdot,\tau)\|^{2}_{L^{2}(\Omega)}+\|\nabla \tilde{w}(\cdot, \tau)\|^{2}_{L^{2}(\Omega)}\leq C
 \Big(\|F_{1}\|_{L^{2}(Q)}^{2}+ T \displaystyle\int_{0}^{\tau} \|\p_{t}\tilde{w}(\cdot, s)\|^{2}_{L^{2}(\Omega)} ds+ \displaystyle\int_{0}^{\tau} \|\nabla \tilde{w}(\cdot, t)\|_{L^{2}(\Omega)}^{2}\,dt
 \Big).$$
Using  Gr\"onwall's Lemma, we obtain 
$\|\p_{t}\tilde{w}(\cdot,\tau)\|_{L^{2}(\Omega)}^{2}\leq C \|F_{1}\|^{2}_{L^{2}(Q)}.$
Which yields in view of (\ref{3.13}),
$$\|r(\cdot,t)\|_{L^{2}(\Omega)}\leq C
\|F_{1}\|_{L^{2}(Q)}.$$
Besides, $F_{1}$ can be seen as 
$$F_{1}(x,t)=\int_{0}^{t}g(x,s)\,ds=\frac{1}{i\sigma} \int_{0}^{t}\,g_{0}(x,s)\p_{s}(e^{i\sigma(x\cdot\omega+s)})\,ds.$$
An integration by parts with respect to the parameter  $s$, yields the existence of  a
constant  $C>0$ such that
$$\|r(\cdot,t)\|_{L^{2}(\Omega)}\leq \frac{C}{\sigma}\|\varphi\|_{H^{3}(\R^{n})}.$$
Bearing in mind that  $\|g\|_{L^{2}(0,T;L^{2}(\Omega))}\leq C \|\varphi\|_{H^{3}(\R^{n})}$ and applying the
energy estimate  to  (\ref{3.12})
we get 
$$\|\p_{t}r(\cdot,t)\|_{L^{2}(\Omega)}+\|\nabla r(\cdot,t)\|_{L^{2}(\Omega)}\leq C \|\varphi\|_{H^{3}(\R^{n})}. $$
The proof of the lemma is completed. 
\end{proof}
 Similarly, we show the following statement 
\begin{Lemm}\label{Lemma 3.2}
  Let $\omega\in\mathbb{S}^{n-1}$ and  $\varphi\in \mathcal{C}^{\infty}_{0}(\R^{n})$.  Let $\phi$ be given  by (\ref{3.10}).  Thus, for  all $\sigma>0$, we can construct a solution to  $\mathscr{H}^{*}_{A,q}v=0$ in $Q$
given by this expression
$$v(x,t)=\phi(x,t)\,b(x,t)\,e^{i\sigma(x\cdot \omega+t)}+r(x,t).$$ 
Here $b$ is defined as follows 
$$b(x,t)=\exp \Big(i\int_{0}^{t} \omega\cdot \overline{A}(x+(t-s)\omega,s)\,ds\Big),$$
and  $r$ is the error term that satisfies
$r_{|t=T}=\p_{t}r_{|t=T}=0$ in $\Omega$ and  $r_{|\Sigma}=0$. 
In addition, there exists $C>0$ such that we have
\begin{equation}\label{3.14}
\sigma \|r\|_{L^{2}(Q)}+\|\nabla r\|_{L^{2}(Q)}\leq C \|\varphi\|_{H^{3}(\R^{n})}.
\end{equation}
\end{Lemm}
\subsection{Stability  for the magnetic potential} This section is dedicated to proving the stable determination of the \textit{known-divergence} magnetic potential $A$  from the DN map $\tilde{\Lambda}_{A,q}$. We consider two mangnetic potentials $A_{1},A_{2}\in \mathcal{A}(M,A_{0})$ and two electric potentials $q_{1},\,q_{2}\in \mathcal{Q}(M,q_{0})$. We denote by 
$A(x,t)=(A_{1}-A_{2})(x,t)$ and  $q(x,t)=(q_{1}-q_{2})(x,t).$ 
Besides, we suppose that 
supp $\varphi\subset \mathscr{C}_r$, then it is clear that 
$$\mbox{Supp}\,\varphi\cap \Omega=\emptyset\,\,\,\,\,\mbox{and}\,\,\,\,\,(\mbox{Supp}\,\varphi \pm T\omega)\cap \Omega=\emptyset,\,\,\,
\forall\,\omega\in \mathbb{S}^{n-1}.$$
Let us consider  a pure imaginary complex vector in $A\in {W}^{3,\infty}(\Omega,\mathbb{C}^{n})$. Then, Green's formula yields
 \begin{equation}\label{FG}
 \int_{Q}\Delta_{A}u\,\overline{v}\,dx\,dt=\int_{Q}\overline{\Delta_{\overline{A}}v}\,u\,dx+\int_{\Sigma}\Big( (\p_{\nu}+i\nu\cdot A)u\,\overline{v}-\overline{(\p_{\nu}+i\nu\cdot\overline{A})v}\,u\,\Big )d\sigma_x\,dt,
 \end{equation}
 for any $u,\,v\in H^{1}(Q)$ satisfying $\Delta u,\,\Delta\,v\in L^{2}(Q)$.  Since $A_{1}=A_{2}$ on $\Sigma$, we  can extend $A$ to a $H^{1}(\R^{n+1})$ vector field by defining it by zero outside of $Q_r$ and we can extend $q$ to an $L^{\infty}(\R^{n+1})$ function by zero outside $Q$. In the sequel of the text, we denote by $A$ and $q$ these extensions. 
\begin{Lemm}\label{Lemma 3.3}
 Let $\omega\in\mathbb{S}^{n-1}$. There exists  $C>0$ depending on $\Omega$, $T$ and $M$ such that we have 
$$\Big|\int_{\R^{n}}\varphi^{2}(y)\Big[ \exp\Big( -i\int_{0}^{T}\omega\cdot A(y-s\omega,s)\,ds \Big)-1 \Big]\,dy  \Big|\leq C\Big(\sigma^{2}
\|\widetilde{\Lambda}_{A_{2},q_{2}}-\widetilde{\Lambda}_{A_{1},q_{1}}\|+\frac{1}{\sigma}
\Big)\|\varphi\|_{H^{3}(\R^{n})}^{2}.$$
 for any \color{black} $\sigma>0$ sufficiently large.
\end{Lemm}

\begin{proof}
Since  $\mbox{supp}\,\varphi\cap \Omega=\emptyset$, then Lemma \ref{Lemma 3.1} yields the existence of a geometrical solution $u_{2}$ to 
 \begin{equation*}
\mathcal{H}_{A_{2},q_{2}}u=0, \quad \mbox{in} \,Q,\quad
    u_{|t=0}=\p_{t}u_{|t=0}=0, \quad\mbox{in}\, \Omega,  
\end{equation*}
having this form
$u_{2}(x,t)=\varphi(x+t\omega)b_{2}(x,t)e^{i\sigma(x\cdot \omega+t)}+r_{2}(x,t),$
where $b_2$ is given by this expression $$b_{2}(x,t)=\exp\Big(i\displaystyle\int_{0}^{t}\omega\cdot A_{2}(x+(t-s)\omega,s)\,ds \Big),$$ and the error term $r_{2}$ obeys (\ref{3.11}). On the other hand, we set $f=u_{2|\Sigma}$ and we denote by  $u_{1}$ the solution of
$$ \mathcal{H}_{A_{1},q_{1}}u_{1}=0,\quad  \mbox{in} \,Q, $$
  with initial data
  $ u_{1|t=0}=\p_{t}u_{1|t=0}=0\,\, \, \mbox{in}\, \Omega, $
  and the following boundary condition  
 $   u_{1}=u_{2}:=f\,\, \mbox{on} \,\Sigma.$
 Putting $u=u_{1}-u_{2}$. Then, $u$ is a solution to 
\begin{equation}\label{3.16}
\left\{
\begin{array}{ll} \mathcal{H}_{A_{1},q_{1}}u=2iA\cdot \nabla u_{2}+V_{A} u_{2}-q u_{2}, & \mbox{in} \,Q, \\
   \\
    u(\cdot,0)=\p_{t}u(\cdot,0)=0, & \mbox{in}\, \Omega,\\
    \\
    u=0, & \mbox{on} \,\Sigma,
  \end{array}
\right.
\end{equation} 
with $A=A_{1}-A_{2},$  $q=q_{1}-q_{2}$ and $V_{A}=A_{2}^{2}-A_{1}^{2}$. Next, from Lemma \ref{Lemma 3.2} and since  $(\mbox{supp} \varphi \pm T\omega)\cap \Omega=\emptyset$, one can see that there exists a geometrical optic solution $v$ to
\begin{equation*}
 \mathcal{H}^{*}_{A_{1},q_{1}}v=0,  \mbox{in} \,Q, \quad
    v_{|t=T}=\p_{t}v_{t=T}=0, \mbox{in}\, \Omega, 
\end{equation*} 
having this form  
$v(x,t)=\varphi(x+t\omega) b_{1}(x,t)e^{i\sigma(x\cdot\omega+t)}+r_{1}(x,t).$ Here  $\mathcal{H}^{*}_{A_{1},q_{1}}v= (\p_{t}^{2}-\Delta_{\overline{A}_{1}}+q_{1}(x,t))v$, $b_1$ is given by  $$b_{1}(x,t)=\exp\big( i\displaystyle\int_{0}^{t}\omega\cdot \overline{A}_{1}(x+(t-s)\omega,s)\,ds\Big),$$ and the error term   $r_{1}$ obeys (\ref{3.14}). We multiply (\ref{3.16}) by $\overline{v}$ and integrate by parts, then from  (\ref{FG}), one gets
\begin{eqnarray}\label{3.17}
\int_{Q}2iA(x,t)\cdot \nabla u_{2}(x,t)\overline{v}(x,t)\,dx\,dt&=&-\int_{\Sigma}(\widetilde{\Lambda}_{A_{1},q_{1}}-\widetilde{\Lambda}_{A_{2},q_{2}})(f) \overline{v}(x,t)\,d\sigma_x\,dt \cr
&&-\int_{Q}(V_{A}(x,t)-q(x,t))u_{2}(x,t)\overline{v}(x,t)\,dx\,dt.
\end{eqnarray} 
We replace the solutions $u_{2}$ and $v$ by their forms, we obtain 
\begin{eqnarray}\label{3.18}
\displaystyle\int_{Q}2iA(x,t)\cdot \nabla u_{2}(x,t)\overline{v}(x,t)\,dx\,dt&=&\displaystyle\int_{Q}2i A(x,t)\cdot \nabla (\phi b_{2})(x,t) (\phi\overline{b_{1}})(x,t)\,dx\,dt\cr
&&+\displaystyle\int_{Q}2iA(x,t)\cdot \nabla (\phi b_{2})(x,t)\overline{r_{1}}(x,t) e^{i\sigma(x\cdot\omega+t)}\,dx\,dt\cr
&&
-2\sigma \displaystyle\int_{Q}\omega\cdot A(x,t) (\phi b_{2})(x,t) (\phi\overline{b_{1}})(x,t)\,dx\,dt\cr
&&-2\sigma \displaystyle\int_{Q} \omega\cdot A(x,t) (\phi b_{2})(x,t)\overline{r_{1}}(x,t) e^{i\sigma(x\cdot \omega+t)}\,dx\,dt\cr
&&+2i\displaystyle\int_{Q}A(x,t)\cdot \nabla r_{2}(x,t)(\phi\overline{b_{1}})(x,t) e^{-i\sigma(x\cdot \omega+t)}\,dx\,dt\cr
&&+2i\displaystyle\int_{Q}A(x,t)\cdot \nabla r_{2}(x,t) \overline{r_{1}}(x,t)\,dx\,dt\cr
&=&-2\sigma\displaystyle\int_{Q}\omega\cdot A(x,t) \phi^2(x,t) b_A(x,t) \,dx\,dt+I_{\sigma},
\end{eqnarray}
where $b_A(x,t)=(b_2\overline{b_1})(x,t)=\exp\big(-i \displaystyle\int_{0}^t \omega\cdot A(y-s\omega,s)\,ds\big)$. Taking $\sigma$ large enough, we get
\begin{equation}\label{3.19}
\|u_{2}\overline{v}\|_{L^{1}(Q)}\leq C \|\varphi\|^2_{H^{3}(\R^{n})},\quad \mbox{and}\quad |I_{\sigma}|\leq C \|\varphi\|^2_{H^{3}(\R^{n})}.
\end{equation}
 Applying the trace theorem, one gets
\begin{eqnarray*}
\Big|\displaystyle\int_{\Sigma}(\widetilde{\Lambda}_{A_{2},q_{2}}-\widetilde{\Lambda}_{A_{1},q_{1}})(f)\,\overline{v}(x,t)\,d\sigma_x\,dt\Big|&\leq&
\|\widetilde{\Lambda}_{A_{2},q_{2}}-\widetilde{\Lambda}_{A_{1},q_{1}}\|\,\|f\|_{H^{1}(\Sigma)}\|\overline{v}\|_{L^{2}(\Sigma)}\cr
\cr
&\leq& \|\widetilde{\Lambda}_{A_{2},q_{2}}-\widetilde{\Lambda}_{A_{1},q_{1}}\|\,\|u_{2}-r_{2}\|_{H^{2}(Q)}\,\|\overline{v}-\overline{r_{1}}\|_{H^{1}(Q)}\cr
\cr
&\leq& C\sigma^{3}\|\widetilde{\Lambda}_{A_{2},q_{2}}-\widetilde{\Lambda}_{A_{1},q_{1}}\|\,\|\varphi\|^2_{H^{3}(\R^{n})},
\end{eqnarray*}
This combined with(\ref{3.18}) and (\ref{3.19}) give 
\begin{eqnarray*}
\Big|\displaystyle\int_{Q}\omega\cdot A(x,t) \phi^2(x,t) b_A(x,t)dx\,dt\Big|\leq C\Big(\sigma^{2}\|\widetilde{\Lambda}_{A_{2},q_{2}}-\widetilde{\Lambda}_{A_{1},q_{1}}\| +\frac{1}{\sigma} \Big)\,\|\varphi\|^2_{H^{3}(\R^{n})}.
\end{eqnarray*}
Bearing in mind that \color{black} $A=0$ outside $Q$, \color{black} and setting $y=x+t\omega$, $s=t-s$, we obtain for any  $\varphi\in \mathcal{C}^{\infty}_{0}(\mathscr{C}_r),$
\begin{eqnarray*}
\Big|\displaystyle\int_{0}^{T}\int_{\R^{n}}\omega\cdot A(y-t\omega,t) \varphi^2(y) b_A(y,t) \,dy\,dt\Big|
\leq C\Big(\sigma^{2}\|\widetilde{\Lambda}_{A_{2},q_{2}}-\widetilde{\Lambda}_{A_{1},q_{1}}\| +\frac{1}{\sigma} \Big)\|\varphi\|^2_{H^{3}(\R^{n})}.
\end{eqnarray*}
Using the fact that
\begin{multline*}
\int_{0}^{T}\!\int_{\R^{n}}\omega\cdot A (y-t\omega,t)\varphi^{2}(y)
b_A(y,t)\,dy\,dt
=i\displaystyle\int_{0}^{T}\!\int_{\R^{n}}\!\!\varphi^{2}(y)
\frac{d}{dt} b_A(y,t)\,dy\,dt\cr 
=i\int_{\R^{n}}\varphi^{2}(y)\Big(b_A(y,T)-1\Big)\,dy,
\end{multline*}
we obtain this estimation 
$$\Big|\int_{\R^{n}}\varphi^{2}(y)\Big( b_A(y,T)-1 \Big)\,dy  \Big|\leq C\Big(\sigma^{2}
\|\widetilde{\Lambda}_{A_{2},q_{2}}-\widetilde{\Lambda}_{A_{1},q_{1}}\|+\frac{1}{\sigma}
\Big)\|\varphi\|_{H^{3}(\R^{n})}^{2}.$$
The proof of the lemma is completed.
\end{proof}
\subsubsection{Estimate for the light-ray transform}
 For any $\omega\in\mathbb{S}^{n-1}$ and $f\in L^{1}(\R^{n+1})$, we define   the light-ray transform  of $f$ that we denote  $\mathcal{R}(f)$ as follows 
 
$$
\mathcal{R}(f)(x,\omega):=\int_{\R}f(x-t\omega,t)\,dt.
$$

 Our next objective  is to find an estimate that relats the light-ray transform of  $(\omega\cdot A)$ to the  DN map $\widetilde{\Lambda}_{A,q}$ on a precise subdomain. Using the previous result, one can estimate  the light-ray transform of $(\omega\cdot A)$ as given in this lemma:
\begin{Lemm}\label{Lemma 3.4}
  There exist positive constants $C,$ $\delta$, $\beta$ and $\sigma_{0}$ such that for any  $\omega\in
\mathbb{S}^{n-1},$ we have
$$|\mathcal{R}(\omega\cdot A)(y,\omega)|\leq C \Big( \sigma^{\delta}\|\widetilde{\Lambda}_{A_{2},q_{2}}-\widetilde{\Lambda}_{A_{1},q_{1}}\|+\frac{1}{\sigma^{\beta}}  \Big)
,\,\,\,\,\,\,\mbox{a.e.}\, y\in\R^{n},$$
for all $\sigma\geq\sigma_{0}$, where $C$ is depending depends only on $T$, 
$M$ and $\Omega$.
\end{Lemm}
\begin{proof}{}
We denote by  $\psi\in\mathcal{C}_{0}^{\infty}(\R^{n})$  a non-negative function that staisfies supp $\psi\subset B(0,1)$ and 
$\|\psi\|_{L^{2}(\R^{n})}=1$. We denote by $\varphi_h$ the mollifier function given by
\begin{equation}\label{3.20}
\varphi_{h}(x)=h^{-n/2}\psi\Big(\frac{x-y}{h}\Big),
\end{equation} for any  $y\in\mathcal{A}_{r}$. Therefore, for $h>0$ small small enough,  these identities hold
$$
\mbox{Supp}\,\varphi_{h}\cap\Omega=\emptyset,\,\,\,\,\,\mbox{and}\,\,\,\,\mbox{Supp}\,\varphi_{h}\pm T\omega\cap \Omega=\emptyset.
$$
On the other hand, we have
\begin{multline}\label{3.21}
\Big| b_A(y,T) -1 \Big|=\Big|
\displaystyle\int_{\R^{n}}\varphi_{h}^{2}(x)\Big( b_A(y,T)-1  \Big)\,dx\Big|\leq\Big| \displaystyle\int_{\R^{n}}\varphi_{h}^{2}(x)\Big(b_A(y,T)
-b_A(x,T)   \Big) \,dx\Big|\cr
 +\Big| \displaystyle\int_{\R^{n}} \varphi_{h}^{2}(x)\Big(b_A(x,T)-1 \Big)
 dx\Big|.
\end{multline}
Thus, using the fact that 
$$\begin{array}{lll} \Big|b_A(y,T)-b_A(x,T)\Big|  \leq C\,\Big|\displaystyle\int_{0}^{T}
i\omega\cdot A(y-s\omega,s) -i\omega\cdot A(x-s\omega,s)ds\Big|,
\end{array}$$
and since $$
\Big|\displaystyle\int_{0}^{T} \Big( i\,\omega\cdot A(y-s\omega,s)-i\, \omega\cdot A(x-s\omega,s)\Big) ds\Big|\leq C \,|y-x|,
$$
one gets from Lemma \ref{Lemma 3.3} with $\varphi=\varphi_{h}$ 
$$
\Big|b_A(y,T)-1
\Big|\leq C\int_{\R^{n}}\varphi_{h}^{2}(x)\,|y-x|\,dx+C\Big(
\sigma^{2}\|\widetilde{\Lambda}_{A_{2},q_{2}}-\widetilde{\Lambda}_{A_{1},q_{1}}\|+\frac{1}{\sigma}
\Big)\|\varphi_{h}\|^{2}_{H^{3}(\R^{n})}.
$$
Next, as the following estimates hold $$
\|\varphi_{h}\|_{H^{3}(\R^{n})}\leq C h^{-3}\,\,\,\,\,\,\,\mbox{and}\,\,\,\,\int_{\R^{n}}\varphi_{h}^{2}(x)|y-x|\,dx \leq C h,
$$
then, we obtain 
$$
 \Big|b_A(y,T)-1
\Big|\leq C \,h+C\Big(
\sigma^{2}\|\widetilde{\Lambda}_{A_{2},q_{2}}-\widetilde{\Lambda}_{A_{1},q_{1}}\|+\frac{1}{\sigma}
\Big)h^{-6}.
 $$
Now we select  $h$ small so that $h=1/\sigma\, h^{6}$, we find two positive constants $\delta$ and $\beta$
such that
$$\Big| b_A(y,T)-1 \Big|\leq
C \Big[
\sigma^{\delta}\|\widetilde{\Lambda}_{A_{2},q_{2}}-\widetilde{\Lambda}_{A_{1},q_{1}}\|+\frac{1}{\sigma^{\beta}}
\Big].$$
 Bearing in mind that  $A$ is a pure imaginary complex vector and that $|X|\leq e^{M}\,|e^{X}-1|$ for any real vector $X\in\R^n$ satisfying $|X|\leq M$, one gets 
$$\Big| -i\int_{0}^{T}\omega\cdot A(y-s\omega,s)\,ds \Big|\leq e^{M T}\Big| b_A(y,T)-1
\Big|.$$
Thus, we have for any $y\in \mathscr{C}_{r}$ and
$\omega\in\mathbb{S}^{n-1}$ 
$$\Big|\int_{0}^{T}i\,\omega\cdot A(y-s\omega,s)\,ds \Big|\leq C \Big(
\sigma^{\delta}\|\widetilde{\Lambda}_{A_{2},q_{2}}-\widetilde{\Lambda}_{A_{1},q_{1}}\|+\frac{1}{\sigma^{\beta}}
\Big).$$
 Now using the fact that  $A=A_{2}-A_{1}=0$ outside $\mathcal{I}_r^*$,
 \color{black} then for any $y\in
 \mathcal{A}_{r}$, and $\,\omega\in\mathbb{S}^{n-1}$, one gets 
\begin{equation}\label{3.22}
\left|\int_{\R}i\,\omega\cdot A(y-t\omega,t)\,dt\right|\leq\,C\,\Big(\sigma^{\delta}\|\widetilde{\Lambda}_{A_{2},q_{2}}-\widetilde{\Lambda}_{A_{1},q_{1}}\|+\frac{1}{\sigma^{\beta}}\Big).
\end{equation}
Besides, even if  $y\in B(0,r/2)$ then, $ |y-t\omega|\geq|t|-|y|\geq
|t|-\displaystyle\frac{r}{2}.$ This yields $(y-t\omega,t)\notin
\mathcal{F}_r$ for $t>r/2$. Moreover, as
$(y-t\omega,t)\notin \mathcal{F}_r$ if $t\leq
\displaystyle\frac{r}{2}$. Then, we have  $(y-t\omega,t)\notin
\mathcal{F}_r\supset \mathcal{I}_r^*$ for $t\in\R.$ This combined this with 
$A=A_{2}-A_{1}=0$ outside $\mathcal{I}_r^*$ \color{black} yields that for any $y\in B(0,r/2)$, one gets 
$$
A(y-t\omega,t)=0,\quad \forall t\in\R,\,\,\,\forall \omega\in\mathbb{S}^{n-1}.$$
Similarly, we show that for $|y|\geq T-r/2$,  we have 
$(y-t\omega,t)\notin \mathcal{B}_r\supset \mathcal{I}_r^*$ for $t\in\R$. Therefore,
$ A(y-t\omega,t)=0$. And we obatin this estimate
\begin{equation}\label{3.23}
\int_{\R}\omega\cdot A(y-t\omega,t)\,dt=0,\,\,\,\,\mbox{a.e.}\,\,y\notin \mathscr{C}_r,\,\,\,\omega\in
\mathbb{S}^{n-1}.
\end{equation}
From (\ref{3.22}) and (\ref{3.23}) we complete the proof and we find out 
$$
|\mathcal{R}(\omega\cdot A)(y,\omega)|=\left|\int_{\R}\omega\cdot A(y-t\omega,t)\,dt\right|\leq\,
C\,\Big(\sigma^{\delta}\|\widetilde{\Lambda}_{A_{2},q_{2}}-\widetilde{\Lambda}_{A_{1},q_{1}}\|
+\frac{1}{\sigma^{\beta}}\Big),
\,\,\,\mbox{a.e}.\,\,y\in\R^{n},\,\,\,\omega\in\mathbb{S}^{n-1}.
$$
This finishes the proof of the lemma.
\end{proof}
\subsubsection{The stability estimate}
 Let us  define  the Fourier transform $\widehat{G}$ of a function $G\in
L^{1}(\R^{n+1})$ as follows
$$
\widehat{G}(\xi,\tau)=\int_{\R}\int_{\R^{n}}G(x,t){e^{-ix\cdot\xi}}e^{-it\tau}\,dx\,dt.
$$
\begin{Lemm}\label{Lemma 3.5}
There exist positive constants  $C,\,\delta,\,\beta,$  and $\sigma_{0}$ such that we have 
\begin{equation}\label{3.24}
| \omega\cdot \widehat{A}(\xi,\tau) | \leq C\,\Big(\sigma^{\delta}\|\widetilde{\Lambda}_{A_{2},q_{2}}-\widetilde{\Lambda}_{A_{1},q_{1}}\|
+\frac{1}{\sigma^{\beta}}\Big), 
\end{equation}
for all $\sigma>\sigma_{0}$ and   $\omega\in\mathbb{S}^{n-1}$ satisfying $\omega\cdot \xi=\tau$\color{black}. 
 \end{Lemm}

\begin{proof}
 Putting 
$x=y-t\omega$, we obtain  for any $\xi\in\R^{n}$ and
$\omega\in\mathbb{S}^{n-1}$ 
$$\begin{array}{lll}
\displaystyle\int_{\R^{n}}\mathcal{R}(\omega\cdot A)(y,\omega)\,{e^{-iy\cdot\xi}}\,dy&=&\displaystyle\int_{\R^{n}}\displaystyle\Big(\int_{\R}\omega\cdot A(y-t\omega,t)\,dt\Big)\,
{e^{-iy\cdot\xi}}\,dy\\
&=&\displaystyle\int_{\R}\int_{\R^{n}}\omega\cdot A(x,t)\,e^{-ix\cdot\xi}{e^{-it\omega\cdot\xi}}\,dx\,dt\\
&=&\omega\cdot\widehat{ A}{(\xi,\omega\cdot\xi)}
=\omega\cdot\widehat{ A}(\xi,\tau).
\end{array}$$
Here $(\xi,\tau)=(\xi,\omega\cdot\xi).$ 
 Since for all $t\in\R$, 
 $\mbox{Supp}\,\, A(\cdot,t)\subset\Omega\subset B(0,r/2)$,
 then we have 
 $$
 \int_{\R^{n}\cap B(0,\frac{r}{2}+T)}\mathcal{R}(\omega\cdot A)(\omega,y)\,{e^{-iy\cdot\xi}}\,dy=\omega\cdot\widehat{ A}(\xi,\tau), \qquad \tau=\omega\cdot \xi.
 $$
Thus, from Lemma \ref{Lemma 3.4}, one gets 
\begin{equation*}
|\omega\cdot \widehat{ A}(\xi,\tau)|\leq C\,\Big(\sigma^{\delta}\|\widetilde{\Lambda}_{A_{2},q_{2}}-\widetilde{\Lambda}_{A_{1},q_{1}}\|+\frac{1}{\sigma^{\beta}}\Big),
\end{equation*}
for any  $\xi\in\R^{n}$, $\omega\in\mathbb{S}^{n-1}$ such that  $\omega\cdot\xi=\tau$, and the proof is complete.
\end{proof}
\noindent Let us define the following set 
$$E=\Big\{(\xi,\tau)\in \R^n\setminus \{0_{R^n}\}\times \R,\,\,|\tau|\leq \frac{1}{2}|\xi|\Big\}.$$ 
We denote 
$$\widehat{\beta}_{j,k}(\xi,\tau)=\xi_j \widehat{a}_k(\xi,\tau)-\xi_k \widehat{a}_j(\xi,\tau),\qquad j,k=1,...,n.$$
\begin{Lemm}\label{Lemma 3.6}
There exist positive constants $C,$ $\delta$, $\beta$ and $\sigma_0$ such that we have 
$$|\widehat{\beta}_{j,k}(\xi,\tau)|\leq C \, \Big(\sigma^{\delta}\|\widetilde{\Lambda}_{A_{2},q_{2}}-\widetilde{\Lambda}_{A_{1},q_{1}}\|
+\frac{1}{\sigma^{\beta}}\Big),$$
for any $\sigma>\sigma_0$ and $(\xi,\tau)\in E\cap B(0,1)$.

\end{Lemm}
\begin{proof}
Let $(\xi,\tau)\in E\cap B(0,1)$ and $(e_k)_k$ be the canonical basis of $\R^n$. For $j,k=1,...,n,\,\,\,j\neq k$, we set
$$\zeta_{j,k}=\frac{\xi_k \,e_j-\xi_j\,e_k}{|\xi_k \,e_j-\xi_j\,e_k|}.$$
Let us consider a unit vector $\omega\in \mathbb{S}^{n-1}$ in this form $$\omega=\frac{\tau}{|\xi|^2}\,\,\xi+\sqrt{1-\frac{\tau^2}{|\xi|^2}}\,\, \zeta_{k,j},$$
with $j,k=1,...,n,\,\,\,j\neq k.$ One can easily see that $\omega\cdot \xi=\tau$. Since div$_{x}$ $(A_1-A_2)=0$ then we have 
$$\omega\cdot \widehat{A}(\xi,\tau)=\sqrt{1-\frac{\tau^2}{|\xi|^2}}\,\,\frac{\widehat{\beta}_{j,k}(\xi,\tau)}{|\xi_k\,e_j-\xi_j\,e_k|}.$$
 Thus, in light of Lemma \ref{Lemma 3.5},   one gets
 $$|\widehat{\beta}_{j,k}(\xi,\tau)|\leq \frac{2|\xi|}{\sqrt{1-\frac{\tau^2}{|\xi|^2}}}\,\,|\omega\cdot \widehat{A}(\xi,\tau)|\leq \frac{4}{\sqrt{3}} \, \Big(\sigma^{\delta}\|\widetilde{\Lambda}_{A_{2},q_{2}}-\widetilde{\Lambda}_{A_{1},q_{1}}\|
+\frac{1}{\sigma^{\beta}}\Big).$$
This completes the proof of the lemma.
\end{proof}
\noindent In what follows 
for $\rho
>0$ and $\kappa\in (\mathbb{N}\cup\{0\})^{n+1}$, we put the following notations
 $$
 |\kappa|=\kappa_{1}+...+\kappa_{n+1},\,\,\,\,\,\,\,\,\,B(0,\rho)=\{x\in\R^{n+1},\,\,|x|<\rho\}.
 $$
Let us recall the
following analytic result, which follows from \cite{[B&I]}.
\begin{Lemm}\label{Lemma 3.7}
Let $\mathcal{O}\subset B(0,1)\subset
\R^{d}$ be a non empty open set,  for $n\geq2$. We consider  an analytic function $G$ in $B(0,2),$ that
satisfies
$$\|\p^{\kappa}G\|_{L^{\infty}(B(0,2))}\leq \frac{M|\kappa|!}{(2\rho)^{|\kappa|}},\,\,\,\,\forall\,\kappa\in(\mathbb{N}\cup\{0\})^{n},$$
for some  $M>0$,  $\rho>0$ and $N=N(\rho)$. Then, we have
$$\|G\|_{L^{\infty}(B(0,1))}\leq N M^{1-\gamma}\|G\|_{L^{\infty}(\mathcal{O})}^{\gamma},$$
where $\gamma\in(0,1)$ depends on $d$, $\rho$ and $|\mathcal{O}|$.
\end{Lemm}
\noindent Let  $\alpha>0$ be fixed.  Setting 
$G_\alpha(\xi,\tau)=\widehat{\beta}_{j,k}(\alpha(\xi,\tau))$, for any 
 $(\xi,\tau)\in\R^{n+1}.$ It is clear that $G_\alpha$ is analytic and
we have
\begin{multline*}
|\p^{\kappa} G_
{\alpha}(\xi,\tau)|=\left|\p^{\kappa}
 \widehat{\beta}_{j,k}(\alpha(\xi,\tau))\right|=\left|\p^{\kappa}\displaystyle\int_{\R^{n+1}} \beta_{j,k}(x,t)\,{e^{-i\alpha
(t,x)\cdot(\xi,\tau)}}\,dx\,dt\right|\cr
=\left|\displaystyle\int_{\R^{n+1}} \beta_{j,k}(x,t)(-i)^{|\kappa|}\alpha^{|\kappa|}(x,t)^{\kappa}{e^{-i\alpha(x,t)\cdot(\xi,\tau)}}\,dx\,dt\right|.
\end{multline*}
This yields that
\begin{multline*}
|\p^{\kappa} G_ {\alpha}(\xi,\tau)|\leq \displaystyle\int_{\R^{n+1}}| \beta_{j,k}(x,t)| \alpha^{|\kappa|}(|x|^{2}+t^{2})^{\frac{|\kappa|}{2}}\,dx\,dt\cr
\leq\|\beta_{j,k}\|_{L^{1}(\mathcal{I}_r^*)}\,\,
\alpha^{|\kappa|}\,\,(2T^{2})^{\frac{|\kappa|}{2}}
\leq C \,\,\displaystyle\frac{|\kappa|!}{(T^{-1})^{|\kappa|}}\,\,e^{\alpha}.
\end{multline*}
Thus, from  Lemma \ref{Lemma 3.7} with $M=Ce^{\alpha}$, $2\rho=T^{-1},$
and $\mathcal{O}=\mathring{E}\cap B(0,1)$, there exists  $\mu\in(0,1)$ such that  for any $(\xi,\tau)\in B(0,1)$, we have 
$$
|G_\alpha(\xi,\tau)|=|\widehat{\beta}_{j,k}(\alpha(\xi,\tau))|\leq C e^{\alpha(1-\gamma)}\|G_\alpha\|_{L^{\infty}(\mathcal{O})}^{\gamma}.
$$
Bearing in min that  $\alpha\,\mathring{E}=\{\alpha(
\xi,\tau),\,(\xi,\tau)\in\mathring{E}\}=\mathring{E}$, we get for any 
$(\xi,\tau)\in B(0,\alpha)$
\begin{eqnarray*}
| \widehat{\beta}_{j,k}(\xi,\tau)|=|G_\alpha(\alpha^{-1}(\xi,\tau))|&\leq& C e^{\alpha(1-\gamma)}\,\|G_\alpha\|_{L^{\infty}(\mathcal{O})}^{\gamma}\cr
&\leq& C e^{\alpha(1-\gamma)} \|\,\widehat{\beta}_{j,k}\,\|^{\mu}_{L^{\infty}(B(0,\alpha)\cap\mathring{E})}\cr
&\leq&C e^{\alpha (1-\gamma)}\|\,\widehat{\beta}_{j,k}\,\|_{L^{\infty}(\mathring{E})}^{\gamma}.
\end{eqnarray*}
The objective  now is to find out a link between  $\beta_{j,k}$ and the DN map. Then we will decompose  the $H^{-1}$ norm of $\beta_{j,k}$ as follows
\begin{multline*}
\|\beta_{j,k}\|_{H^{-1}(\R^{n+1})}^{2/\gamma}=\Big(\displaystyle\int_{|(\tau,\xi)|<\alpha}(1+|(\tau,\xi)|^{2})^{-1}
|\widehat{\beta}_{j,k}|^{2}\, d\xi d\tau
+\int_{|(\xi,\tau)|\geq\alpha}(1+|(\tau,\xi)|^{2})^{-1}|\widehat{\beta}_{j,k}|^{2}\, d\xi d\tau\,\Big)^{1/\gamma}\cr
\leq  C\Big(\alpha^{n+1}\|\widehat{\beta}_{j,k}\|^{2}_{L^{\infty}(B(0,\alpha))}+\,\alpha^{-2}\|\beta_{j,k}\|^{2}_{L^{2}(\R^{n+1})}\Big)^{1/\gamma}.
\end{multline*}
Hence, from the previous lemma one gets
$$\begin{array}{lll}
\|\beta_{j,k}\|^{2/\gamma}_{H^{-1}(\R^{n+1})}&\leq&\,C\displaystyle\Big(\alpha^{{n+1}}\,e^{2\alpha(1-\gamma)}\,(\sigma^{\delta}
\|\widetilde{\Lambda}_{A_{2},q_{2}}-\widetilde{\Lambda}_{A_{1},q_{1}}\|
+\frac{1}{\sigma^{\beta}})^{
2\gamma}+\alpha^{-2}\Big)^{1/\gamma}\\
&\leq&C\displaystyle\Big(\alpha^{\frac{n+1}{\gamma}} \,e^{\frac{2\alpha(1-\gamma)}{\gamma}}
\sigma^{2\beta}\|\widetilde{\Lambda}_{A_{2},q_{2}}-\widetilde{\Lambda}_{A_{1},q_{1}}\|^{2}
+\alpha^{\frac{n+1}{\gamma}}\,e^{\frac{2\alpha(1-\gamma)}{\gamma}}\,\sigma^{-2\beta}+\alpha^{-2/\gamma}\Big).
\end{array}$$
Let us consider $\alpha_{0}>0$ be  large enough. Suppose that $\alpha>\alpha_{0}$.
Putting 
$$
\sigma=\alpha^{\frac{n+3}{2\mu\gamma}}\,e^{\frac{\alpha(1-\mu)}{\mu\gamma}}.
$$
Since   $\alpha>\alpha_{0},$ we have 
$\sigma>\sigma_{0}$ and
$\alpha^{\frac{n+1}{\gamma}}\,e^{\frac{2\alpha(1-\gamma)}{\gamma}}\,\sigma^{-2\beta}=\alpha^{-2/\gamma}.$
Which yields 
 $$\begin{array}{lll} \|\beta_{j,k}\|^{2/\gamma}_{H^{-1}(\R^{n+1})}&\leq&
C\Big(\alpha^{\frac{\beta(n+1)+\delta(n+3)}{\beta\gamma}}\,e^{\frac{2\alpha(\beta+\delta)(1-\gamma)}{\beta\gamma}}
\|\widetilde{\Lambda}_{A_{2},q_{2}}-\widetilde{\Lambda}_{A_{1},q_{1}}\|^{2}
+\alpha^{-2/\gamma}\Big)\\
&\leq&
C\,\displaystyle\Big(e^{N\alpha}\|\widetilde{\Lambda}_{A_{2},q_{2}}-\widetilde{\Lambda}_{A_{1},q_{1}}\|^{2}+\alpha^{-2/\gamma}\Big),
\end{array}$$
where $N$ depends on $\delta,\,\beta,\,n,$ and $\gamma$. By  minimizing  the right hand-side of the above estimation with respect to $\alpha$ we get for 
$0<\|\widetilde{\Lambda}_{A_{2},q_{2}}-\widetilde{\Lambda}_{A_{1},q_{1}}\|<m$,  the following estimate 
 $$
 \alpha=\frac{1}{N}|\,\log\|\widetilde{\Lambda}_{A_{2},q_{2}}-\widetilde{\Lambda}_{A_{1},q_{1}}\|\,|.
 $$
 This entails that 
\begin{eqnarray}\label{3.25}
\|\beta_{j,k}\|_{H^{-1}(\R^{n+1})}&\leq&
C\Big(\|\widetilde{\Lambda}_{A_{2},q_{2}}-\widetilde{\Lambda}_{A_{1},q_{1}}\|+|\,\log\|\widetilde{\Lambda}_{A_{2},q_{2}}-\widetilde{\Lambda}_{A_{1},q_{1}}\|\,|^{-2/\gamma}\Big)^{\gamma/2}\cr
&\leq&C\Big(\|\widetilde{\Lambda}_{A_{2},q_{2}}-\widetilde{\Lambda}_{A_{1},q_{1}}\|^{\gamma/2}+|\log\|\widetilde{\Lambda}_{A_{2},q_{2}}-\widetilde{\Lambda}_{A_{1},q_{1}}\||^{-1}\Big).
\end{eqnarray}
Thus, one gets 
$$\|\mbox{curl}_x(A_1)-\mbox{curl}_x(A_2)\|_{H^{-1}(\R^{n+1})}\leq C\Big(\|\widetilde{\Lambda}_{A_{2},q_{2}}-\widetilde{\Lambda}_{A_{1},q_{1}}\|^{\gamma/2}+|\log\|\widetilde{\Lambda}_{A_{2},q_{2}}-\widetilde{\Lambda}_{A_{1},q_{1}}\||^{-1}\Big).$$
Let us consider  $\delta>1$ and $k\in\mathbb{N}$ such that $k=n/2+2\delta$. By interpolation, it follows that 
\begin{multline*}
\|\mbox{curl}_x(A_1)-\mbox{curl}_x(A_2)\|_{L^{\infty}(\R^{n+1})}\leq C\|\mbox{curl}_x(A_1)-\mbox{curl}_x(A_2)\|_{H^{\frac{n}{2}+\delta}(\R^{n+1})}\cr
\leq C \,\|\mbox{curl}_x(A_1)-\mbox{curl}_x(A_2)\|_{H^{-1}(\R^{n+1})}^{1-\eta}\,\|\mbox{curl}_x(A_1)-\mbox{curl}_x(A_2)\|_{H^{k}(\R^{n+1})}^{\eta}\cr
\leq C\,\|\mbox{curl}_x(A_1)-\mbox{curl}_x(A_2)\|_{H^{-1}(\R^{n+1})}^{1-\eta}.
\end{multline*}
 Therefore, one can find $\mu\in (0,1)$  such that 
$$  \|\mbox{curl}_x(A_1)-\mbox{curl}_x(A_2)\|_{L^{\infty}(\R^{n+1})}\leq C\Big(\|\widetilde{\Lambda}_{A_{2},q_{2}}-\widetilde{\Lambda}_{A_{1},q_{1}}\|+|\log\|\widetilde{\Lambda}_{A_{2},q_{2}}-\widetilde{\Lambda}_{A_{1},q_{1}}\||^{-1}\Big)^{\mu}. $$
Then for all $p>n$, one have 
$$ \|\mbox{curl}_x(A_1)-\mbox{curl}_x(A_2)\|_{L^{\infty}((0,T),L^{p}(\Omega))}\leq C\Big(\|\widetilde{\Lambda}_{A_{2},q_{2}}-\widetilde{\Lambda}_{A_{1},q_{1}}\|+|\log\|\widetilde{\Lambda}_{A_{2},q_{2}}-\widetilde{\Lambda}_{A_{1},q_{1}}\||^{-1}\Big)^{\mu}. $$
From  Lemma B.5 in \cite{[Ibtissem]} and using the fact that div$_{x}$ ($A$)= div$_{x}$ $(A_1-A_2)=0$ we  obtain 
$$\|A\|_{L^{\infty}((0,T),W^{1,p}(\Omega))}\leq C\Big(\|\widetilde{\Lambda}_{A_{2},q_{2}}-\widetilde{\Lambda}_{A_{1},q_{1}}\|+|\log\|\widetilde{\Lambda}_{A_{2},q_{2}}-\widetilde{\Lambda}_{A_{1},q_{1}}\||^{-1}\Big)^{\mu}.$$
As a consequence, we have 
$$\|A\|_{L^{\infty}(\mathcal{I}_r^*)}\leq C\Big(\|\widetilde{\Lambda}_{A_{2},q_{2}}-\widetilde{\Lambda}_{A_{1},q_{1}}\|+|\log\|\widetilde{\Lambda}_{A_{2},q_{2}}-\widetilde{\Lambda}_{A_{1},q_{1}}\||^{-1}\Big)^{\mu}.$$
\subsection{Stability for the electric potential}\label{subsection 3.3}
In this section, we aim to show that the time-dependent electric potential can be stably recovered from the Dirichlet to Neumann map $\widetilde{\Lambda}_{A,q}$.   Let us show  the following  preliminary statement. 
\begin{Lemm}\label{Lemma 3.8} There exist postive constants $C$, $\delta$, and $\sigma_{0}$ such that we  have 
$$
 \Big| \mathcal{R}(q)(y,\omega)\Big|\leq  C
 \Big(\sigma^{\delta}\|\widetilde{\Lambda}_{A_{2},q_{2}}-\widetilde{\Lambda}_{A_{1},q_{1}}\|+\sigma^{\delta}\|A\|_{L^{\infty}(\R^{n+1})}
+\frac{1}{\sigma^{\beta}}    \Big).
$$
for  all $\sigma>\sigma_{0},$ where $C$ is depending only on $T$, $\Omega$ and $M$.
\end{Lemm}
\begin{proof}
From (\ref{3.17}) one can see that 
$$\begin{array}{lll}
\displaystyle\int_{Q}\!\!\!&q(x,t)&\!\!\!u_{2}(x,t)\overline{v}(x,t)\,dx\,dt=\displaystyle\int_{\Sigma}(\widetilde{\Lambda}_{A_{1},q_{1}}-\widetilde{\Lambda}_{A_{2},q_{2}})(f) \overline{v}(x,t)\,d\sigma_x\,dt \\
&&+\displaystyle\int_{Q}2iA\cdot \nabla u_{2}(x,t)\overline{v}(x,t)\,dx\,dt+\displaystyle\int_{Q}V_{A}(x,t)u_{2}(x,t)\overline{v}(x,t)\,dx\,dt,
\end{array}$$
with  $V_{A}=(A_{2}^{2}-A_{1}^{2})$. We replace  $u_{2}$ and $v$ by their expressions,  we find out that  
$$\begin{array}{lll}
\displaystyle\int_{Q}q(x,t)u_{2}(x,t)\overline{v}(x,t)\,dx\,dt&=&\displaystyle\int_{Q}q(x,t)\phi^{2}(x,t)b_A(x,t)\,dx\,dt\\
&&+\displaystyle\int_{Q}q(x,t) \phi
(x,t)b_{2}(x,t)e^{i\sigma(x\cdot\omega+t)}\overline{r_{1}}(x,t)\,dx\,dt\\
&&+\displaystyle\int_{Q}q(x,t) {\phi}(x,t)\overline{b_{1}}(x,t)e^{-i\sigma(x\cdot\omega+t)}r_{2}(x,t)\,dx\,dt\\
&&+\displaystyle\int_{Q}q(x,t)r_{2}(x,t)\,\overline{r_{1}}(x,t)\,dx\,dt.
\end{array}$$
Thus , we get 
\begin{equation}\label{3.26}
\int_{Q}q(x,t)\phi^{2}(x,t)b_A(x,t)\,dx\,dt=\displaystyle\int_{\Sigma}(\widetilde{\Lambda}_{A_{1},q_{1}}-\widetilde{\Lambda}_{A_{2},q_{2}})(f) \overline{v}(x,t)\,d\sigma_x\,dt+I_{\sigma},
\end{equation}
where $I_{\sigma}$ denotes the following  quantity
$$\begin{array}{lll}
I_{\sigma}&=&\displaystyle\int_{Q}2iA\cdot \nabla u_{2}(x,t)\overline{v}(x,t)\,dx\,dt-\displaystyle\int_{Q}V_{A}(x,t)u_{2}(x,t)\overline{v}(x,t)\,dx\,dt\\
&&-\displaystyle\int_{Q}q(x,t)\phi(x,t)b_{2}(x,t)e^{i\sigma(x\cdot \omega+t)}\overline{r}_{1}(x,t)\,dx\,dt-\displaystyle\int_{Q}q(x,t) r_{2}(x,t) \overline{r_{1}}(x,t)\,dx\,dt
\\
&&-\displaystyle\int_{Q}q(x,t){\phi}(x,t) \overline{b_{1}}(x,t)e^{-i\sigma(x\cdot \omega+t)}r_{2}(x,t)\,dx\,dt.
\end{array}$$
For $\sigma$ large enough, we obtain 
\begin{equation}\label{3.27}
|I_{\sigma}|\leq C \Big(\frac{1}{\sigma}+\sigma \|A\|_{L^{\infty}(\mathcal{I}_r^*)} \Big)\|\varphi\|^2_{H^{3}(\R^{n})}.
\end{equation}
 From the trace theorem, we get 
\begin{equation}\label{3.28}
\Big| \displaystyle\int_{\Sigma}(\widetilde{\Lambda}_{A_{1},q_{1}}-\widetilde{\Lambda}_{A_{2},q_{2}})(f) \overline{v}(x,t)\,d\sigma_x\,dt   \Big|\leq C\sigma^{3}\|\widetilde{\Lambda}_{A_{1},q_{1}}-\widetilde{\Lambda}_{A_{2},q_{2}}\|\|\varphi\|^2_{H^{3}(\R^{n})}.
\end{equation}
Therefore, using  (\ref{3.26}), (\ref{3.27}) and (\ref{3.28}) one gets  
 \begin{multline*}
 \Big| \int_{Q} q(x,t)\phi^{2}(x,t)b_A(x,t)\,dx\,dt \Big|
 \leq C\Big( \sigma^{3}\|\widetilde{\Lambda}_{A_{1},q_{1}}-\widetilde{\Lambda}_{A_{2},q_{2}}\|+\sigma\|A\|_{L^{\infty}(\R^{n+1})}+\frac{1}{\sigma}  \Big)\|\varphi\|^2_{H^{3}(\R^{n})}.
 \end{multline*}
Bearing in mind that  $q=0$ outside $Q$, then by setting $y=x+t\omega$, $s=t-s$ we obtain 
\begin{multline}
\Big|\displaystyle\int_{0}^{T}\int_{\R^{n}}q(y-t\omega,t)\varphi^{2}(y)b_A(y,t)\,dy\,dt   \Big|\cr
\leq C\Big( \sigma^{3}\|\widetilde{\Lambda}_{A_{1},q_{1}}-\widetilde{\Lambda}_{A_{2},q_{2}}\|+\sigma\|A\|_{L^{\infty}(\R^{n+1})}+\displaystyle\frac{1}{\sigma}  \Big)\|\varphi\|^{2}_{H^{3}(\R^{n})}.
\end{multline}
This combined with $$\begin{array}{lll}
\Big|\displaystyle\int_{0}^{T}\int_{\R^{n}}q(y-t\omega,t)\varphi^{2}(y)\Big(1-b_A(y,t) \Big)  \,dy\,dt  \Big|
\leq C\|A\|_{L^{\infty}(\R^{n+1})}\|\varphi\|_{H^{3}(\R^{n})}^{2},
\end{array}$$
yields the following estimation  
$$\Big|\int_{0}^{T}\int_{\R^{n}}q(y-t\omega,t)\varphi^{2}(y)\,dy\,dt \Big|\leq \Big( \sigma^{3}\|\widetilde{\Lambda}_{A_{1},q_{1}}-\widetilde{\Lambda}_{A_{2},q_{2}}\|+\sigma\|A\|_{L^{\infty}(\R^{n+1})}+\displaystyle\frac{1}{\sigma}  \Big)\|\varphi\|^{2}_{H^{3}(\R^{n})}.$$
Next, we  consider  the sequence  $(\varphi_{h})_{h}$ defined in the previous section with
$y\in \mathscr{C}_r$. We get 
\begin{multline*}
\Big|\displaystyle\int_{0}^{T}q(y-t\omega,t) dt\Big|=\Big|\displaystyle\int_{0}^{T}\int_{\R^{n}}q(y-t\omega,t)\,\varphi_{h}^{2}(x)\,dx\,dt
\Big|\cr
\leq \Big| \displaystyle\int_{0}^{T}\!\!\!\int_{\R^{n}}\!\!q(x-t\omega,t)\varphi_{h}^{2}(x)dx\,dt\Big|
+\Big| \displaystyle\int_{0}^{T}\int_{\R^{n}}\Big(q(y-t\omega,t)-q(x-t\omega,t)\Big)\varphi_{h}^{2}(x)dx\,dt \Big|.
\end{multline*}
Thus,  bearing in mind that 
$|q(y-t\omega,t)-q(x-t\omega,t)|\leq C |y-x|$, one gets 
\begin{multline*}
\Big|\displaystyle\int_{0}^{T}q(y-t\omega,t)dt\Big| \leq
C\Big(\sigma^{3}\|\widetilde{\Lambda}_{A_{1},q_{1}}-\widetilde{\Lambda}_{A_{2},q_{2}}\|+\sigma\|A\|_{L^{\infty}(\R^{n+1})}+\displaystyle\frac{1}{\sigma}
\Big)
\|\varphi_{h}\|^{2}_{H^{3}(\R^{n})}\cr
+C\displaystyle\int_{\R^{n}}\!\!|x-y|\varphi_{h}^{2}(x)dx.
\end{multline*}
Now using the fact that  $\|\varphi_{h}\|_{H^{3}(\R^{n})}\leq C h^{-3}$
 and  $ \displaystyle\int_{\R^{n}} |x-y|
\varphi_{h}^{2}(x)\,dx\leq \,C \,h$, we thus deduce that 
$$
\Big|\int_{0}^{T}q(y-t\omega,t)\,dt\Big|\leq C \Big( \sigma^{3}\|\widetilde{\Lambda}_{A_{2},q_{2}}-\widetilde{\Lambda}_{A_{1},q_{1}}\|
+\sigma\|A\|_{L^{\infty}(\R^{n+1})} +\frac{1}{\sigma}\Big)h^{-6}+C\,h.
$$ Choosing  $h$ small enough so that   $h=h^{-6}/\sigma$, we we find out  $\delta>0$ and $\beta>0$ such that
$$
\Big|\int_{0}^{T}q(y-t\omega,t) \Big|\leq C
\Big(\sigma^{\delta}\|\widetilde{\Lambda}_{A_{2},q_{2}}-\widetilde{\Lambda}_{A_{1},q_{1}}\|+\sigma^{\delta}\|A\|_{L^{\infty}(\R^{n+1})}
+\frac{1}{\sigma^{\beta}}    \Big).
$$
Since  $q=q_{2}-q_{1}=0$
outside of $\mathcal{I}_r^*$, we then deduce that for any 
 $y\in\mathcal{A}_{r}$ and $\omega\in\mathbb{S}^{n-1}$,
 $$
 \Big| \int_{\R}q(y-t\omega,t)\,dt \Big|\leq  C
 \Big(\sigma^{\delta}\|\widetilde{\Lambda}_{A_{2},q_{2}}-\widetilde{\Lambda}_{A_{1},q_{1}}\|+\sigma^{\delta}\|A\|_{L^{\infty}(\R^{n+1})}
+\frac{1}{\sigma^{\beta}}    \Big).
$$
Next, by arguing  as before and using the fact that $q=0$ outside $\mathcal{I}_r^*$, we conclude that for all $y\in\R^{n}$,
 $$
 \Big| \mathcal{R}(q)(y,\omega)\Big|\leq  C
 \Big(\sigma^{\delta}\|\widetilde{\Lambda}_{A_{2},q_{2}}-\widetilde{\Lambda}_{A_{1},q_{1}}\|+\sigma^{\delta}\|A\|_{L^{\infty}(\R^{n+1})}
+\frac{1}{\sigma^{\beta}} \Big).
$$
This completes the proof of the lemma.
\end{proof}
Our objective  is to find an
estimat for the Fourier transform of $q$. Thus, proceeding as before, we find out: 
\begin{Lemm}\label{Lemma 3.9}
There exist positive constants 
$C$, $\delta$, $\beta$ and $\sigma_{0}$, such that we have 
$$|\widehat{q}(\xi,\tau)|\leq C \Big(\sigma^{\delta}\|\widetilde{\Lambda}_{A_{2},q_{2}}-\widetilde{\Lambda}_{A_{1},q_{1}}\|+\sigma^{\delta}\|A\|_{L^{\infty}(\R^{n+1})}
+\frac{1}{\sigma^{\beta}}\Big),\,\,\,\,\,\,\,\mbox{a.\,e}, (\xi,\tau)\in
E,$$ for all $\sigma>\sigma_{0}$, where $C$ is depending  only on  $T$, $\Omega$ and 
$M$.
\end{Lemm}
Next, from the above estimation as well as the analytic continuation
argument, we upper bound the Fourier transform
of $q$ in a suitable ball $B(0,\alpha)$  as follows
\begin{equation}\label{3.29} |
\widehat{q}(\xi,\tau)|\leq Ce^{\alpha
(1-\gamma)}\Big(\sigma^{\delta}\|\widetilde{\Lambda}_{A_{2},q_{2}}-\widetilde{\Lambda}_{A_{1},q_{1}}\|
+\sigma^{\delta}\|A\|_{L^{\infty}(\R^{n+1})}+\frac{1}{\sigma^{\beta}}
\Big)^{\gamma},
\end{equation}
 for some $\gamma\in(0,1)$ and where $\alpha>0$ is
assumed to be sufficiently large.  
In order to get an estimate
linking  $q$ to the DN measurement 
$\widetilde{\Lambda}_{A_{2},q_{2}}-\widetilde{\Lambda}_{A_{1},q_{1}}, $ we need to upper bound  the
$H^{-1}(\R^{n+1})$ norm of $q$ as follows
$$\begin{array}{lll}
\|q\|^{\frac{2}{\gamma}}_{H^{-1}(\R^{n+1})}\leq C\Big[\alpha^{n+1}\|\widehat{q}\|^{2}_{L^{\infty}(B(0,\alpha))}
+\alpha^{-2}\|q\|^{2}_{L^{2}(\R^{n+1})}   \Big]^{\frac{1}{\gamma}}.
\end{array}$$
Thus, from  (\ref{3.29}), one gets 
\begin{equation}\label{3.30}
\|q\|_{H^{-1}(\R^{n+1})}^{\frac{2}{\gamma}}\leq C \Big[
\alpha^{\frac{n+1}{\gamma}}e^{\frac{2\alpha(1-\gamma)}{\gamma}}\para{\sigma^{2\delta}\epsilon^{2}+\sigma^{2\delta}
\|A\|^{2}_{L^{\infty}(\R^{n+1})}+\sigma^{-2\beta}}+\alpha^{\frac{-2}{\gamma}}
\Big],
\end{equation}
where we have set $\epsilon=\|\widetilde{\Lambda}_{A_{2},q_{2}}-\widetilde{\Lambda}_{A_{1},q_{1}}\|$. From Theorem \ref{Theorem 1.1}, we obtain 
$$
\|q\|^{\frac{2}{\gamma}}_{H^{-1}(\R^{n+1})}\leq C \Big[ \alpha^{\frac{n+1}{\gamma}}e^{\frac{2\alpha(1-\gamma)}{\gamma}}\para{\sigma^{2\delta}\epsilon^{2}
+\sigma^{2\delta}\epsilon^{2\mu_{1}\mu_{2}}+\sigma^{2\delta}|\log \,\epsilon|^{-2\mu_{2}}+\sigma^{-2\beta}}+\alpha^{-\frac{2}{\gamma}}
 \Big],
$$
for some $\gamma,\,\mu_{1},\,\mu_{2}\in(0,1)$ and $\delta,\,\beta>0$. Let us consider 
$\alpha_{0}>0$ large  enough. Taking  $\alpha>\alpha_{0}$ and putting 
$$
 \sigma=\alpha^{\frac{n+3}{2\gamma\beta}} e^{\frac{\alpha(1-\gamma)}{\gamma\beta}}.
 $$
By taking $\alpha>\alpha_{0}$, one can assume that $\sigma>\sigma_{0}$. Thus,  from (\ref{3.30}) we have  
$$
\|q\|^{\frac{2}{\gamma}}_{H^{-1}(\R^{n+1})}\leq C\Big[ e^{N\alpha}\para{\epsilon^{2}+\epsilon^{2s}+|\log\epsilon|^{-2\mu_{2}}}+\alpha^{-\frac{2}{\gamma}} \Big],
$$
for some $s,\,\mu_{1},\,\mu_{2}\in(0,1)$. Here  $N$ depends on $n,\,\gamma,\delta$ and $\beta$. Then, for  $\epsilon$ small enough, one gets 
\begin{equation}\label{3.31}
\|q\|^{\frac{2}{\gamma}}_{H^{-1}(\R^{n+1})}\leq C\Big( e^{N\alpha}|\log
\epsilon|^{-2\mu_{2}}+\alpha^{-\frac{2}{\gamma}}  \Big).
\end{equation}
Choosing $\alpha$ as follows
$$\alpha=\frac{1}{N} \log|\log\epsilon|^{\mu_{2}},$$
where  $\epsilon <c\leq1$. Then, the quantity (\ref{3.31}) yields 
$$\|q\|_{H^{-1}(\mathcal{I}_r^*)}\leq \|q\|_{H^{-1}(\R^{n+1})}\leq C \Big( \log |\log \|\widetilde{\Lambda}_{A_{2},q_{2}}-\widetilde{\Lambda}_{A_{1},q_{1}}\||^{\mu_{2}}
\Big)^{-1}. $$
Now, we consider  $\delta>1$ and $k\in\mathbb{N}$ so that it satisfies $k=\frac{n}{2}+2\delta$. By interpolation, we obtain 
$$\begin{array}{lll}
\|q\|_{L^{\infty}(\R^{n+1})}&\leq&C\|q\|_{H^{\frac{n}{2}+\delta}(\R^{n+1})}\\
&\leq& C \,\|q\|_{H^{-1}(\R^{n+1})}^{1-\eta}\,\|q\|_{H^{k}(\R^{n+1})}^{\eta}\\
&\leq&C\,\|q\|_{H^{-1}(\R^{n+1})}^{1-\eta}.
\end{array}$$
 Therefore, there exists $\mu'\in (0,1)$  such that 
 $$\|q\|_{L^{\infty}(\mathcal{I}_r^*)}\leq C \Big( \log |\log \|\widetilde{\Lambda}_{A_{2},q_{2}}-\widetilde{\Lambda}_{A_{1},q_{1}}\||^{\mu_{2}}
\Big)^{-\mu'}, $$
 which finishes  the proof of Theorem \ref{Theorem 2.2}.
\subsection{Proof of Theorem\ref{Theorem 1.1}}
At this stage we are well prepared to deal with the problem under consideration, that is the identification of $V$ and $p$ appearing in (\ref{1.1}) from the knowledge of $\Lambda_{V,p}$. Based on Lemma 2.1 and Theorem 2.2 one can easily see the desired result.
\section{Determination of the unknown  terms from boundary observation  and final conditions}\label{Section 4}
In this section, we are going to show that the velocity field $V$ and the electric potential $p$ can be stably recovered in a larger region $\mathcal{I}_{r}^\sharp\supset \mathcal{I}_{r}^*$ by enlarging the set of  data. To state things clearly, in the present section,  our observations are given by the following  operator
$$\begin{array}{ccc}
\mathcal{R}_{V,p}:\mathscr{H}^{1}_{0}(\Sigma)&\longrightarrow&\mathcal{K}\\
\,\,\,\,\,\,\,\,\,\,\,f&\longmapsto&(\p_{\nu}u,u(\cdot,T),\p_{t}u(\cdot,T)),
\end{array}$$
associated with the equation (\ref{1.1}) with $(u_{0},u_{1})=(0,0)$. In view of Lemma \ref{Lemma 2.1}, this equivalently amounts to showing  Theorem \ref{Theorem 2.2}. That is to show that  $A$ and $q$ can be determined in  $\mathcal{I}_{r}^\sharp\supset \mathcal{I}_{r}^*$ \color{black} from the knowledge of the equivalent operator $\widetilde{\mathcal{R}}_{A,q}$ given by
$$\begin{array}{ccc}
\widetilde{\mathcal{R}}_{A,q}:\mathscr{H}^{1}_{0}(\Sigma)&\longrightarrow&\mathcal{K}\\
\,\,\,\,\,\,\,\,\,\,\,f&\longmapsto&((\p_{\nu}+i A\cdot \nu)u,u(\cdot,T),\p_{t}u(\cdot,T)),
\end{array}$$
associated with the equation (\ref{2.2}) with $(u_{0},u_{1})=(0,0)$. In order to establish  Theorem \ref{Theorem 1.2}, we shall consider the geometric
optics solutions constructed before with a
function $\varphi$ satisfying supp $\varphi\cap \Omega=\emptyset.$ Note that
this time, we have more flexibility on the support of  $\varphi$
so we don't need to assume that supp $\varphi\pm
T\omega$ and $\Omega$ are disjoint anymore. Let us denote by 
$$
\widetilde{\mathcal{R}}^{1}_{A,q}(f)=(\p_{\nu}+i A\cdot\nu)u,\,\,\,\,\,\widetilde{\mathcal{R}}^{2}_{A,q}(f)=u(\cdot,T),\,\,\,\,\,\,\widetilde{\mathcal{R}}_{A,q}^{3}(f)=\p_{t}u(\cdot,T).
$$ 
Our first aim to stably retrieve the magnetic potential with the known divergence in a larger region of the domain by adding the final data of the solution $u$ of the wave operator. 
\begin{Lemm}\label{Lemma4.1}
 We consider a function 
$\varphi\in\mathcal{C}^{\infty}_{0}(\R^{n})$ satisfying 
$\varphi\cap\Omega=\emptyset$. There exists  a positive constant $C$, such that we have 
$$\Big|\int_{\R^{n}}\varphi^{2}(y)\Big( b_A(y,T)-1\Big) \,dy  \Big|\leq C\Big(\sigma^{2}
\|\widetilde{\mathcal{R}}_{A_{2},q_{2}}-\widetilde{\mathcal{R}}_{A_{1},q_{1}}\|+\frac{1}{\sigma}
\Big)\|\varphi\|_{H^{3}(\R^{n})}^{2},$$  for any
$\omega\in\mathbb{S}^{n-1}$. Here $C$ depends only on  $T$, $\Omega$ and $M$.
\end{Lemm}
\begin{proof}{}
From Lemma \ref{Lemma 3.1}, one can find  a geometrical optic solution
$u_{2}$ to
$$
 \mathcal{H}_{A_2,q_2}u_{2}=0 \quad  \mbox{in}\,\,\,Q, \quad 
   u_{2|t=0}=\p_{t}u_{2|t=0}=0 \quad \mbox{in}\,\,\,\Omega,
$$
having  this form
\begin{equation}\label{4.32}
u_{2}(x,t)=\varphi(x+t\omega)b_2(x,t)e^{i\sigma(x\cdot\omega+t)}+r_{2}(x,t),
\end{equation}
 Here the error term 
$r_2$ obeys(\ref{3.11}) and $b_2$ is defined by
$b_2(x,t)= \exp\Big( i\displaystyle\int_{0}^t  \omega\cdot A_2(x+(t-s)\omega,s)\Big)$.  
Let $u_{1}$ be the solution to  
$\mathcal{H}_{A_1,q_1}u_{1}=0,\,\,  \mbox{in}\,\,\,Q, $
with initial data 
    $u_{1|t=0}=\p_{t}u_{1|t=0}=0\,\, \mbox{in}\,\,\,\Omega$ and the boundary condition $  u_{1}=f:=u_{2|\Sigma}$ on $\Sigma$.
We put  $u=u_{1}-u_2$. Thus,  $u$ solves the  following system
\begin{equation}\label{4.33}
\left\{
  \begin{array}{ll}
   \mathcal{H}_{A_1,q_1}u=2 iA\cdot \nabla u_2+V_A u_2-q u_2 & \mbox{in}\,\,\,Q, \\
\\
    u(\cdot,0)=\p_{t}u(\cdot,0)=0 & \mbox{in}\,\,\,\Omega, \\
\\
    u=0 & \mbox{on}\,\,\,\Sigma,
  \end{array}
\right.
\end{equation}
 with $A=A_{1}-A_{2}$ and $q=q_{1}-q_{2}$. From Lemma \ref{Lemma 3.2}, there exists  a
geometrical optic solution $v$ to the adjoint problem
$
\mathcal{H}^*_{A_1,q_1}v=0 \,\,\mbox{in}\,\,\,Q,
$ in the
form
\begin{equation}\label{4.34}
v(x,t)=\overline{\varphi}(x+t\omega)b_1(x,t)e^{i\sigma(x\cdot\omega+t)}+r_{1}(x,t),
\end{equation}
 where $r_{1}$
satisfies (\ref{3.14}). We multiply  (\ref{4.33}) with $\overline{v}$ and integrate by parts, we find out in  light  of (\ref{FG})
\begin{multline*}\label{EQ4.43}
\int_{Q}2i A(x,t)\cdot \nabla u_2\, \overline{v}\,dx dt=\int_{\Sigma}(\widetilde{\mathcal{R}}^{1}_{A_{2},q_{2}}-\widetilde{\mathcal{R}}^{1}_{A_{1},q_{1}})(f)\overline{v}(x,t)d\sigma_x dt-\int_{\Omega}\!(\widetilde{\mathcal{R}}^{3}_{A_{2},q_{2}}-\widetilde{\mathcal{R}}^{3}_{A_{1},q_{1}})(f)\overline{v}(x,T)\,dx \cr
-\int_{\Omega}(\widetilde{\mathcal{R}}^{2}_{A_{2},q_{2}}-\widetilde{\mathcal{R}}^{2}_{A_{1},q_{1}})(f)\p_t
\overline{v}(x,T)-\int_{Q}(V_A(x,t)-q(x,t))u_2(x,t)\overline{v}(x,t)\,dx\,dt.
\end{multline*}
We  replace $u_2$ and $v$ by their expressions, and we proceed  as in the proof of Lemma \ref{Lemma 3.3}, we use  (\ref{3.14})
and the Cauchy-Schwartz inequality, we obtain
\begin{multline*}
\Big| \int_{\R^n}\varphi^2(y)\Big(b_A(y,T)-1 \Big)\,dy  \Big|
\leq  \frac{C}{\sigma}\Big( \|\widetilde{\mathcal{R}}_{A_2,q_2}-\widetilde{\mathcal{R}}_{A_1,q_1}\|\|u_{2|\Sigma}\|_{H^1(\Sigma)}\|\psi_{\sigma}\|_{\mathcal{K}}+ \|\varphi\|^2_{H^3(\R^n)} \Big),
\end{multline*}
where 
$\psi_{\sigma}=\displaystyle\para{u^{-}_{|\Sigma},\,u^{-}(\cdot,T),\,\p_{t}u^{-}(\cdot,T)}$. 
From the trace theorem and since 
$$\|u_{2|\Sigma}\|_{H^{2}(Q)}\leq C \sigma^2 \|\varphi\|_{H^3(\R^n)}, \quad\mbox{and}\quad  \|\psi_\sigma\|_{H^1(Q)\times L^{2}(\Omega)\times L^2(\Omega)}\leq C \sigma \|\varphi\|_{H^3(\R^n)},$$
one gets 
$$\Big| \int_{\R^n} \varphi^2(y) \,\,\Big( b_A(y,T)-1\Big)\,dy\Big|\leq C\Big(\sigma^{2}
\|\widetilde{\mathcal{R}}_{A_{2},q_{2}}-\widetilde{\mathcal{R}}_{A_{1},q_{1}}\|
+\frac{1}{\sigma}\Big)\|\varphi\|^{2}_{H^{3}(\R^{n})},$$ which completes the proof of the Lemma.  
\end{proof}
 Next, we just need to apply Lemma \ref{Lemma4.1} to the function  $\varphi_{h}$ given by (\ref{3.20}) with $y\notin\Omega$,
 use the fact that   $A=A_{1}-A_{2}=0$ outside $\mathcal{I}_{r}^\sharp$ and  proceed  as in the previous section to end up proving that the magnetic potential $A$ stably depends on the Neumann measurement $\widetilde{\mathcal{R}}_{A,q}$ and we have 
 $$\|A_1-A_2\|_{L^\infty(\mathcal{I}_{r}^\sharp)}\leq C |\log \|\widetilde{\mathcal{R}}_{A_2,q_2}-\widetilde{\mathcal{R}}_{A_1,q_1}\| |^{-\mu_1}\color{black}. $$
Our next focus is to show that the electric potential $q$  can also be stably recovered
in the region $\mathcal{I}_{r}^\sharp$, from the knowledge of the response given by the operator
$\widetilde{\mathcal{R}}_{A,q}$ and of course by the use of the $\log$ type stability estimate we have already obtained for the magnetic potential. 
\begin{Lemm}\label{Lemma4.2}
Let $\omega\in\mathbb{S}^{n-1}$ and $\varphi\in
\mathcal{C}_{0}^{\infty}(\R^{n})$ satisfying $\mbox{supp}\,\varphi\cap\Omega=\emptyset$.  There exists a positive constant
$C$ such that
 we have $$
\Big|  \displaystyle\int_{0}^{T}\!\!\!\displaystyle\int_{\R^{n}} q(y-t\omega,t) \varphi^{2}(y)
\,dy\,dt\Big|
\leq
C\Big(\sigma^{3}\|\widetilde{\mathcal{R}}_{A_{2},q_{2}}-\widetilde{\mathcal{R}}_{A_{1},q_{1}}\|+\sigma\|A\|_{L^{\infty}(\mathcal{I}_{r}^\sharp)}+\displaystyle\frac{1}{\sigma}
\Big)
\|\varphi\|^{2}_{H^{3}(\R^{n})}.
 $$
Here $C$ depends only on $\Omega$, $M$ and $T$.
\end{Lemm}
\begin{proof}{}
From the orthogonality identity one can see that 
\begin{multline*}
\int_{Q} q(x,t)u_2\overline{v}\,dx dt=\int_{\Sigma}(\widetilde{\mathcal{R}}^{1}_{A_{2},q_{2}}-\widetilde{\mathcal{R}}^{1}_{A_{1},q_{1}})(f)\overline{v}(x,t)d\sigma_x dt+\int_{\Omega}(\widetilde{\mathcal{R}}^{3}_{A_{2},q_{2}}-\widetilde{\mathcal{R}}^{3}_{A_{1},q_{1}})(f)\overline{v}(x,T)\,dx \cr
+\int_{\Omega}(\widetilde{\mathcal{R}}^{2}_{A_{2},q_{2}}-\widetilde{\mathcal{R}}^{2}_{A_{1},q_{1}})(f)\p_t
\overline{v}(x,T)
+\int_{Q}(V_A+2 i A\cdot \nabla) u_{2}(x,t)\overline{v}(x,t)dx\,dt.
\end{multline*}
Thus, if we replace  replacing $u_2$ and $v$ by their expressions,  use (\ref{3.22}) and since  $q=q_{2}-q_{1}=0$ outside
$\mathcal{I}_{r}^\sharp$,  we obtain after  making the change of variables $y=x+t\omega$ and  by proceeding   as in the proof of Lemma \ref{Lemma 3.8} the following estimation 
$$
\Big| \int_{0}^{T}\!\!\!\int_{\R^{n}}q(y-t\omega,t)\varphi^{2}(y)\,dy\,dt
 \Big|\leq C \Big( \sigma^{3}\|\widetilde{\mathcal{R}}_{A_{2},q_{2}}-\widetilde{\mathcal{R}}_{A_{1},q_{1}}\|
+\sigma\|A\|_{L^{\infty}(\mathcal{I}_{r}^\sharp)}+\frac{1}{\sigma}
\Big)\|\varphi\|_{H^{3}(\R^{n})}^{2},
$$
 which completes the proof of the lemma.
\end{proof}
The next step is to apply Lemma \ref{Lemma4.2} to the function $\varphi_{h}$ given by  (\ref{3.20}), except this time  
$y\notin\Omega$. Again,  using the fact that  $q=q_{2}-q_{1}=0$ outside $\mathcal{I}_{r}^\sharp$ and
repeating the same arguments of the Subsection \ref{subsection 3.3}, we complete the proof of Theorem \ref{Theorem 2.3}. 

Finally, in view of Lemma \ref{Lemma 2.1}, one can  easily establish the result given in Theorem \ref{Theorem 1.2}.

\section{Determination of the unknown terms from boundary measurements and final data by varying the initial conditions}\label{Section 5}
In this section, we focus on the same inverse problem treated before.  More precisely, our aim here is to extend the identification regions of the unknown coefficients by enlarging the set
of data. Indeed, Our data are now the responses of the medium measured for all
possible initial data.

For
$(V_{i},p_{i})\in\mathcal{C}^{3}(\overline{Q})\times\mathcal{C}^{1}(\overline{Q})$,
$i=1,\,2$, we define $(V,p)=(V_{2}-V_{1},p_{2}-p_{1})$ in $Q$ and
$(V,p)=(0_{\R^{n}},0)$ outside $Q$. We proceed  as in
the proof of Theorem \ref{Theorem 1.1} and Theorem \ref{Theorem 1.2}, we show 
a $\log$-type stability estimate for the recovery  of the velocity field $V$
 everywhere in $Q$ from the
measurement $\Gamma_{V,p}$. We proceed by a similar way for the recovery of  $p$ to end up obtaining an  estimate of 
$\log$-$\log$-type over the whole $Q$.

Note that in order to prove such estimates, we proceed as in the previous sections. We first need to prove Theorem \ref{Theorem 2.4}, but  in this case the
support of the function $\varphi_{h}$ defined by (\ref{3.20}) is more flexible. We do not need to take  any condition on its support.


\end{document}